\documentclass{siamltex}
\usepackage{amsmath} 	    
\usepackage{amsfonts,amssymb}    	    

\usepackage{mathtools}
\DeclarePairedDelimiter{\norm}{\lVert}{\rVert}
\usepackage{wrapfig}
\usepackage{todonotes}

\usepackage[top=3cm, bottom=3cm, left=3cm, right=3cm, heightrounded, marginparwidth=3.1cm, marginparsep=5mm]{geometry}                                          

\newcommand{\wOm}{\widehat{\Omega}}
\newcommand{\Om}[2]{\Omega^{#1}_{#2}}

\newcommand{\Td}[1]{\mathcal{T}_{\delta}^{#1}}
\newcommand{\Odt}[1]{{\Om{#1}{\delta,h}}}
\newcommand{\T}{\mathcal T}

\newcommand{\TS}[1]{\mathcal{T}_{\pm}^{#1}}

\def\Fh{\mathcal{F}_h}

\def\eps{\varepsilon}
\def\G{\Gamma}
\def\Gp{\Gamma_{\!+}}
\def\GQ{\Gamma_{\!\cQ}}
\def\dist{\text{dist}}

\newcommand{\E}[1]{\mathcal{E}_{#1}}
\newcommand{\err}{\mathbb{E}}
\newcommand{\R}{\mathbb{R}}
\newcommand{\consist}{\mathcal{E}_C^n}
\newcommand{\interpol}{\mathcal{E}_I^n}

\newcommand{\cI}{\mathcal{I}}
\newcommand{\bn}{\mathbf{n}}
\newcommand{\bw}{\mathbf{w}}

\newcommand{\by}{\mathbf{y}}

\newcommand{\cQ}{\mathcal{Q}}
\newcommand{\cO}{\mathcal{O}}
\newcommand{\cS}{\mathcal{S}}

\newtheorem{remark}{Remark}[section]
\newtheorem{assumption}{Assumption}

\newcommand{\rev}[1]{{\color{black}#1}}

\begin{document}

\title{Analysis of a finite element method for PDEs in evolving domains with topological changes}

\author{Maxim A. Olshanskii\thanks{Department of Mathematics, University of Houston, Houston, Texas 77204-3008, USA, (maolshanskiy@uh.edu), www.math.uh.edu/\string~molshan} \and Arnold Reusken\thanks{Institut f\"ur Geometrie und Praktische  Mathematik, RWTH-Aachen University, D-52056 Aachen, Germany (reusken@igpm.rwth-aachen.de)} }

\maketitle

\begin{abstract}
The paper  presents the first rigorous error analysis of an unfitted finite element method for a linear parabolic problem posed on an evolving domain $\Omega(t)$ that may undergo  a topological change, such as, for example, a domain splitting. The domain evolution is assumed to be $C^2$-smooth away from a critical time $t_c$, at which the topology may change instantaneously.
To accommodate such topological transitions in the error analysis, we introduce several structural assumptions on the evolution of $\Omega(t)$ in the vicinity of the critical time. These assumptions allow a specific stability estimate even across singularities. Based on this stability result we  derive optimal-order discretization error bounds, provided the continuous solution is sufficiently smooth. 
We demonstrate the applicability of our assumptions with examples of level-set  domains undergoing topological transitions and discuss cases where the analysis fails. The theoretical error estimate is confirmed by the results of a numerical experiment.
\end{abstract}

\section{Introduction} The numerical solution of PDEs posed on time-dependent domains or surfaces is a common problem that arises in many fields of application in engineering and natural sciences. Computational techniques for addressing this challenge range from Lagrangian approaches with domain-fitted meshes to purely Eulerian methods that employ meshes indifferent to the domain's motion. There is extensive literature on these computational techniques, that we do not survey here.

In the context of this paper, it is relevant to distinguish between problems with smoothly evolving domains (or surfaces) that undergo relatively mild deformations, and problems in which the domains are strongly deforming or even undergo topological changes. It is well known that Lagrangian approaches are particularly suitable for the former class of problems, whereas Eulerian techniques are often employed to handle the latter.
Regarding the rigorous error analysis of discretization methods, we note the following. For problems with smoothly evolving domains (or surfaces), there is an extensive body of literature. Error analyses of Lagrangian methods can be found in, e.g., \cite{badia2006analysis,Dziuk07,Formaggia2004,elliott2015error,Kovacs2018,ma2022fourth,li2023optimal,li2025optimal}. Papers addressing error analyses of Eulerian methods for problems on smoothly evolving domains or surfaces include \cite{badia2023space,heimann2023geometrically,alphonse2015abstract,elliott2021unified,burman2022eulerian,lou2022isoparametric,neilan2024eulerian,olshanskii2024eulerian,olshanskii2014error}.
To the best of our knowledge, there is no existing work that provides a rigorous error analysis of a discretization method for a moving domain problem with a sharp boundary representation that undergoes a topological change. Such an analysis is the main focus of this paper.

We consider  a linear parabolic problem posed on a time-dependent domain that may undergo topological changes. For the discretization we apply a finite element method  introduced in \cite{lehrenfeld2019eulerian}, which uses an Eulerian framework and a sharp representation of the physical domain boundary. This  method  embeds the deforming domain into a fixed background mesh and combines a standard time-stepping scheme with an implicit extension of the finite element solution to a narrow band surrounding the physical domain. For the case of a smoothly evolving domain an error analysis of this method is presented in  \cite{lehrenfeld2019eulerian}. The method, however, is also well-suited for solving  PDEs in domains undergoing complex deformations --- including topological changes.

Numerous scenarios can lead to topological changes in two- or three-dimensional evolving domains. The analysis presented in this paper covers only a specific subset of such scenarios. In the following section, we describe the domain evolution in detail and formulate a set of assumptions sufficient to establish an optimal-order error analysis of the method. Later in the paper, we investigate level-set domains that undergo topological changes and for which these assumptions are satisfied. We also consider cases in which some of the assumptions are violated and discuss  limitations of the analysis.
Roughly speaking, our assumptions are as follows: the domain evolution is $C^2$-smooth away from a single critical time $t_c$, where a topological change occurs. The types of topological changes we allow include domain splitting, hole formation, and the disappearance of an interior component (island). 

These assumptions on the domain evolution are sufficient to derive a uniform in time estimate for the variation of the $L^2(\Omega(t))$-norm of a smooth function, which provides control over the implicit narrow band extension of the numerical solution. This control is essential for ensuring the stability and accuracy of the method.  The key estimate for this narrow-band extension is established in Section~\ref{sec_key}. The parabolic  model problem and the finite element formulation are explained in Section~\ref{s:model}.
Section~\ref{SecError} presents the stability and error analysis of the method. The main result, stated in Theorem~\ref{Th2}, establishes optimal-order convergence in the standard “energy” norm.

It is already known from the literature that this type of Eulerian unfitted finite element technique can handle PDEs posed on domains undergoing topological changes in a robust and straightforward manner; see, e.g.,  \cite{lehrenfeld2018stabilized,lehrenfeld2019eulerian,yushutin2020numerical,von2021falling,badia2023space,olshanskii2025conservative}. Therefore, in Section~\ref{s:numer} on  numerical results we restrict to one simple example which confirms the optimal convergence rates for a domain splitting problem with a prescribed smooth solution.

\section{Domain evolution} \label{Sectdomains} 
For $t \in [0,T]$, consider a time-dependent bounded domain $\Omega(t) \subset \mathbb{R}^d$, where $d = 2,3$, with $\Gamma(t) = \partial\Omega(t)$. We assume that there exists a \emph{critical time} $t_c \in (0,T)$ such that the evolution of $\Omega(t)$ is $C^2$-smooth before and after $t_c$ (i.e., any possible topological change is \emph{instantaneous}) in the following sense. With $\mathcal{I}_1 := [0,t_c)$, $\mathcal{I}_2 := (t_c, T]$, we assume that
there exist $C^2$-smooth bounded domains $\Omega_i^0 \subset \mathbb{R}^d$, $i=1,2$ (each consisting of a finite number of connected components), and diffeomorphisms $\Phi_i(t): \Omega_i^0 \to \Omega(t)$, $t \in \mathcal{I}_i$, such that $\Phi_i \in C^2(\overline{\Omega_i^0} \times \mathcal{I}_i; \mathbb{R}^d)$, for $i=1,2$.

Thus, $t = t_c$ is a singular time at which we allow the $C^2$-norm of one or both mappings to become unbounded. At the same time, we assume a continuous transition of $\Omega(t)$ through the singularity in the sense of measure:

\begin{assumption}\label{Ass0}  
 It holds	
	\[
	\lim\limits_{\delta\to\pm0} \mathrm{meas}_d\big(\Omega\rev{(t_c)}\, \triangle \,\Omega(t_c+\delta)\big) = 0.
	\]
\end{assumption}  

Such a piecewise smooth domain evolution as described above allows for rather general space--time singularities, including those resulting from instantaneous domain merging, domain splitting or the creation of holes. Some of these scenarios will be covered by our analysis, while others will be not.
To distinguish between these cases, we  introduce additional assumptions on the evolution of $\Omega(t)$ in the vicinity of $t = t_c$. We start with motivating these.

For $t \neq t_c$  the domain  $\Omega(t)$ is  $C^2$ smooth with a well defined outward pointing unit normal vector on $\Gamma(t)$, denoted by $\bn_\G$. The normal velocity of $\Gamma(t)$ is  $V_\G=\partial_t\Phi_i\cdot\bn_\G$, $t\in \mathcal{I}_i$. We introduce 
\[ \Gp(t) = \{x\in \Gamma(t): \, V_\Gamma(x,t)>0\},
\]
and $[f]_+:=\max\{f,0\}$. For a sufficiently smooth function $u\,:\,\Omega(t)\to\R$ such that $\partial_t u=0$ the Reynolds transport theorem yields 
\begin{equation} \label{eq1}
	\frac{d}{dt}\|u\|_{\Omega(t)}^2 =\frac{d}{dt} \int_{\Omega(t)} u^2 \, dx = \int_{\Gamma(t)} V_\Gamma u^2 \, ds \leq \int_{\Gamma(t)} [V_\Gamma]_+ u^2 \, ds \leq \sup_{\G(t)}\,[V_\G]_+\, \|u\|_{\Gp(t)}^2.
\end{equation}
Here and further we use the notation $\|\cdot\|_\omega$ for the $L^2$-norm over a domain $\omega$.
 Applying the trace inequality $\|u\|_{\Gp(t)}^2 \leq  C_{\Gamma}(t)\|u\|_{\Omega(t)}\|u\|_{H^1(\Omega(t))}$  we obtain, 
\begin{equation} \label{eq2}
	\frac{d}{dt}\|u\|_{\Omega(t)}^2 \leq C_{\Gamma}(t)\sup_{\G(t)}\,[V_\G]_+\, \|u\|_{\Omega(t)}\|u\|_{H^1(\Omega(t))}. 
\end{equation}
If the factor $C_{\Gamma}(t) \sup_{\G(t)}\,[V_\G]_+$ is uniformly bounded in time, then \eqref{eq2}
yields control of the rate of change in $\|u\|_{\Omega(t)}^2$ in terms of a weighted sum of  $\|u\|_{\Omega(t)}^2$ and $\|\nabla u\|_{\Omega(t)}^2$:
\begin{equation} \label{eq2b}
		\frac{d}{dt}\|u\|_{\Omega(t)}^2 \leq C_0\|u\|_{\Omega(t)}\|u\|_{H^1(\Omega(t))} \leq C_0\, \Big((1+\frac{1}{2\eps})\|u\|_{\Omega(t)}^2 + 2\eps \|\nabla u\|_{\Omega(t)}^2\Big)\,\quad \forall\,\eps>0,
\end{equation}
with  some $C_0>0$,   
for all sufficiently smooth $u$ that satisfy  $\partial_t u=0$.

A discrete in time analogue of \eqref{eq2b} is crucial for applying a Gronwall type argument and serves as the basis for several stability and error analyses, e.g.  in \cite{lehrenfeld2019eulerian}.  In particular, the narrow band estimate derived in Section~\ref{sec_key} is an estimate of the form \eqref{eq2b} for evolving domains which undergo certain types of topological changes. However, in general the term $C_{\Gamma}(t) \sup_{\G(t)}\,[V_\G]_+$ in \eqref{eq2} may blow up for $t \to t_c$ precluding a uniform control as in \eqref{eq2b}. This motivates the assumptions~\ref{Ass1} and~\ref{Ass3} below.  A further  motivation of these assumptions is given in Remark~\ref{RemNew}.
 
 \smallskip
 \begin{assumption}\label{Ass1}
  Assume  that 
 \begin{equation}\label{ass1}
V_{\max}^+ :=\sup_{t\in\mathcal{I}_1\cup \mathcal{I}_2} \sup_{\G(t)}\,[V_\G]_+ < \infty.
 \end{equation}
 \end{assumption}

Consider the trace inequality,
\begin{equation}\label{est:tr}
	\| u \|_{L^2(\Gp(t))}^2 \le C_{\Gamma}(t) \| u \|_{L^2(\Omega(t))}\| u \|_{H^1(\Omega(t))},\quad t\in\mathcal{I}_1\cup \mathcal{I}_2.
\end{equation} 

\begin{assumption}\label{Ass3}
	For the optimal  constant from \eqref{est:tr} assume
		\begin{equation}\label{ass3} 
			\sup\limits_{t\in\mathcal{I}_1\cup \mathcal{I}_2}  C_{\Gamma}(t)< \infty.
		\end{equation}
\end{assumption}

It is important to note that in  \eqref{est:tr} the trace is taken  on $\Gp(t)$ and not on $\Gamma(t)$.

We further need the notion of  space--time domains: 
\[
\cQ = \bigcup\limits_{t \in (0,T)} \Omega(t) \times \{t\}, \quad \cQ_i = 
\bigcup\limits_{t \in \cI_i} \Omega(t) \times \{t\},~~i=1,2.
\]
The `lateral' boundary of $\cQ$ is denoted by   \[\GQ:= \bigcup\limits_{t \in (0,T)} \Gamma(t)\times \{t\} \subset \R^{d+1}.\]

In the error analysis in Section~\ref{SecError} we use extensions of functions from the evolving physical domain $\Omega(t)$ to a (small) spatial neighborhood. This extension needs to be uniformly in time stable in certain Sobolev norms.  Such uniform stability will be achieved here by employing a suitable extension  from the space--time domain $\cQ$ to $\R^{d+1}$. 
This brings us to the following assumption:

\smallskip
\begin{assumption}\label{Ass4}
We assume $\cQ$ is a Lipschitz domain.
\end{assumption}
\smallskip

In the Sections~\ref{Sectnonsmooth} and \ref{newSectdomains} we study scenarios of domains with topological changes 
for which the assumptions above are fulfilled. In particular, the analysis in Section~\ref{newSectdomains} yields examples and counterexamples of evolving domains satisfying Assumptions~\ref{Ass1}--\ref{Ass4}. 
 

\section{Topological change at a weakly singular point} \label{Sectnonsmooth}

In this section, we consider a transition scenario where the singularity occurs  instantaneous at $t=t_c$ and at a single point $(x_c,t_c)$ in space–time, where locally certain additional conditions are satisfied. We will verify that Assumptions~\ref{Ass1} and~\ref{Ass3} then hold. More precisely, we assume that the inverse diffeomorphisms $\Phi_i(t)^{-1}$, for $t \in \cI_i$, $i=1,2$, have continuous extensions as $t \to t_c$, away from the critical point $x_c$:
\begin{align*}
	\Phi_1(t_c)^{-1}(x)& = \lim_{t\uparrow t_c} \Phi_1(t)^{-1}(x), \quad x \in \Omega\rev{(t_c)} \setminus\{x_c\}, \\
	\Phi_2(t_c)^{-1}(x)& = \lim_{t\downarrow t_c} \Phi_2(t)^{-1}(x), \quad x \in \Omega\rev{(t_c)} \setminus\{x_c\}.
\end{align*}

For $z \in \R^{d+1}$ and $\delta >0$, we denote the ball centered at $z$ with radius $\delta$ by $B_\delta (z)$.
\begin{definition} \label{defweak} 
	\rm A point $(x_c,t_c) \in \GQ$ is called a \emph{weakly singular point} if the following two conditions hold:
	\begin{equation} \label{condition1}
		\forall~\eps>0:\quad	\Phi_i^{-1}\in C^2\Big(\overline{\cQ_i\setminus  B_\eps (x_c,t_c)};\, \R^{d+1}\Big),\quad i=1,2,
	\end{equation}
	and there exists $\delta >0$ such that
	\begin{equation} \label{weakly}
		V_\Gamma(x,t) < 0 \quad \text{for all } (x,t) \in \GQ^- \cap B_\delta(x_c,t_c), 
	\end{equation}
	where $\GQ^- := \GQ \setminus \{(x_c,t_c)\}$.
\end{definition}

Condition~\eqref{condition1} ensures that the domain evolution is $C^2$-smooth everywhere except at the critical point. In particular, this implies that $V_\Gamma(x,t)$ is well-defined and finite for all $(x,t) \in \GQ^-$. Hence, the sign condition~\eqref{weakly} is meaningful. In Section~\ref{newSectdomains} we explain why the sign condition  $V_\Gamma < 0$ near the singular point is natural. 
We also show that for certain classes of level set domains, this condition is in a certain sense necessary to be able to control the variation of $\|u\|_{\Omega(t)}^2$ as a function of $t$.
\smallskip
 
We now show that conditions~\eqref{condition1} and~\eqref{weakly} are sufficient for Assumption~\ref{Ass1}.

\begin{lemma} \label{Cor1}
	If conditions~\eqref{condition1} and~\eqref{weakly} are satisfied, then $\|[V_\Gamma]_+\|_{L^\infty(\GQ)} < \infty$, where $[V_\Gamma]_+ := \max\{V_\Gamma, 0\}$. Hence, Assumption~\ref{Ass1} is fulfilled.
\end{lemma}

\begin{proof}
	Let $B_\delta := B_\delta(x_c,t_c)$. Due to~\eqref{condition1}, we estimate:
	\begin{align*}
		\|V_\Gamma\|_{L^\infty(\GQ \setminus B_\delta)} 
		&= \max_{i=1,2} \left\| \partial_t \Phi_i \circ \Phi_i^{-1} \right\|_{L^\infty(\GQ \setminus B_\delta)} \\
		&\leq \max_{i=1,2} \|\Phi_i\|_{W^{1,\infty}(\Phi_i^{-1}(\GQ \setminus B_\delta))} \\
		&\leq c \max_{i=1,2} \|\Phi_i^{-1}\|_{W^{1,\infty}(\GQ \setminus B_\delta)} < \infty.
	\end{align*}
	Using $V_\Gamma < 0$ on $\GQ^- \cap B_\delta$, we conclude:
	\[
	\|[V_\Gamma]_+\|_{L^\infty(\GQ)} \leq \|V_\Gamma\|_{L^\infty(\GQ \setminus B_\delta)} < \infty.
	\]
\end{proof}

To verify Assumption~\ref{Ass3}, we derive a trace estimate. Recall the definition $\Gp(t) := \{\, x \in \Gamma(t) \mid V_\Gamma(x,t) > 0\,\}$.

\begin{lemma} \label{lemtrace}
	Assume conditions~\eqref{condition1} and~\eqref{weakly} hold. Then there exists a constant $C_{tr}$, independent of $t$, such that for all $t \in \cI_1 \cup \cI_2$:
	\begin{equation} \label{ass2check}
		\|u\|_{L^2(\Gp(t))}^2 \leq C_{tr} \|u\|_{\Omega(t)} \|u\|_{H^1(\Omega(t))}, \quad \text{for all } u \in H^1(\Omega(t)).
	\end{equation}
\end{lemma}

\begin{proof}
	Let $t \in \cI_1$, and take a ball $B_\delta$ as in Definition~\ref{defweak}. Define $\tilde \Gamma := \Gamma(t) \setminus B_\delta$. Then $\Gp(t) \subset \tilde \Gamma$, and $\tilde \Gamma$ is $C^2$-smooth.
		Let $\phi_0$ be a mollified signed distance function for $\Omega_1^0$ such that $\phi_0 \in C^2(\rev{\overline{\Omega_1^0}})$, $\phi_0 = 0$ on $\partial \Omega_1^0$, and $|\nabla \phi_0| > 0$ on $\partial \Omega_1^0$. Define $\phi = \phi_0 \circ \Phi_1^{-1} : \Omega(t) \to \R$.
	
	From~\eqref{condition1} and the definition of $\phi$, we find constants $c_1 > 0$ and $c_2$, independent of $t \in \cI_1$, such that
	\begin{equation} \label{h7}
		|\nabla \phi(x)| \geq c_1, \quad \|\nabla^2 \phi(x)\|_2 \leq c_2 \quad \text{for all } x \in \tilde \Gamma.
	\end{equation}
	Define $\bn(z) := \frac{\nabla \phi(z)}{|\nabla \phi(z)|}$ for $z \in \tilde \Gamma$, and for $\eps > 0$, let
	\[
	U_\eps := \{\, x = z + \xi \bn(z) \mid z \in \tilde \Gamma,\ -\eps \leq \xi \leq 0 \,\}.
	\]
	For $\eps > 0$ sufficiently small (independent of $t$), we \rev{verify below that $U_\eps \subset \overline{\Omega(t)}$} and $U_\eps \cap B_{\frac{1}{2}\delta} = \emptyset$. Moreover, $(z,\xi)$ uniquely parametrizes $x = z + \xi \bn(z) \in U_\eps$. There is a constant $c_3$ independent of $t\in\cI_1$ such that $\sup_{x \in U_\eps} \|\nabla^2 \phi(x)\|_2 \leq c_3$ holds.
	Using Taylor expansion, we obtain for $x = z + \xi \bn(z)$ and suitable $\eta \in (-\eps, 0)$:
	\begin{align*}
		\phi(x) &= \phi(z) + \xi \nabla \phi(z) \cdot \bn(z) + \tfrac{1}{2} \xi^2 \bn(z)^T \nabla^2 \phi(z + \eta \bn(z)) \bn(z) \\[0.3ex]
		&= \xi |\nabla \phi(z)| \left(1 + \tfrac{1}{2} \xi \frac{\bn(z)^T \nabla^2 \phi(z + \eta \bn(z)) \bn(z)}{|\nabla \phi(z)|} \right).
	\end{align*}
	Using $\frac{|\bn(z) \cdot \nabla^2 \phi (z+ \eta \bn(z))\bn(z)|}{|\nabla \phi(z)|} \leq \frac{c_3}{c_1}$ it follows that for $\xi$ that satisfies $- \min\{\eps,\frac{2c_1}{c_3}\} \leq \xi \leq 0$ we have $\phi(x) \leq 0$. Without loss of generality we assume $ \min\{\eps,\frac{2c_1}{c_3}\}=\eps$ and thus $U_\eps \subset \overline{\Omega(t)}$ holds for this $\eps$ that is independent of $t$.
	
	We now proceed with a standard trace estimate for $u \in C^1(\Omega(t))$. Define $\tilde u(z, \xi) := u(z + \xi \bn(z))$. Then
	\[
	\frac{\partial \tilde u}{\partial \xi}(z, \xi) = \bn(z) \cdot \nabla u(z + \xi \bn(z)).
	\]
	For $z \in \tilde \Gamma$, we compute:
	\[
	u(z)^2 = \tilde u(z,0)^2 = \int_{-\eps}^0 \frac{\partial}{\partial \xi} \big( (1 + \tfrac{\xi}{\eps}) \tilde u(z,\xi)^2 \big)\, d\xi.
	\]
	Integrating over $\tilde \Gamma$ and applying the coarea formula yields:
	\begin{align*}
		\|u\|_{L^2(\Gp(t))}^2 &\leq \|u\|_{L^2(\tilde \Gamma)}^2 \leq \tfrac{1}{\eps} \|u\|_{L^2(U_\eps)}^2 + 2 \|u\|_{L^2(U_\eps)} \|\nabla u\|_{L^2(U_\eps)} \\
		&\leq \sqrt{2} \max\left\{ \tfrac{1}{\eps}, 2 \right\} \|u\|_{L^2(\Omega(t))} \|u\|_{H^1(\Omega(t))}.
	\end{align*}
	The same argument applies for $t \in \cI_2$, completing the proof.
\end{proof}

This result confirms that, under a topological change occurring at a weakly singular point $(x_c, t_c)$, the uniform trace estimate required by Assumption~\ref{Ass3} is satisfied.


\section{A class of level set domains with topological changes} \label{newSectdomains} 
For further insight into possible scenarios of topological transitions and analysis of Assumptions~\ref{Ass0}--\ref{Ass4}, we study cases where $\Omega(t)$ is characterized as the subzero level set of a Lipschitz continuous level set function $\phi$:
\begin{equation} \label{subzero} 
	\forall ~t \in [0,T]:~~\Omega(t)=\{\, x \in \R^d ~|~ \phi(x,t) < 0\,\},\quad \Gamma(t)=\{\, x \in \R^d ~|~ \phi(x,t) = 0\,\}. 
\end{equation}

Note that for continuous $\phi$, Assumption~\ref{Ass0} is always satisfied. One of the main reasons for the widespread use of level set methods is that a \emph{smoothly} evolving level set function can implicitly represent a topological change. In this section, we restrict ourselves to this smooth case. More precisely, we assume
\begin{equation} \label{eq:smooth} 
	\phi \in C^2(\R^d \times [0,T]). 
\end{equation}
In Section~\ref{secnonsmooth}, we briefly address examples involving less regular level set functions.

If $|\nabla \phi(x,t)| \geq c_0 > 0$ for all $x \in \Gamma(t)$ and $t \in [0,T]$, then, by the implicit function theorem, $\Gamma(t)$ does not undergo topological changes.
To allow for topological changes, we consider an isolated \emph{critical point} $(x_c,t_c) \in \GQ$, i.e.,
\begin{equation} \label{critpoint}
	\nabla \phi(x_c,t_c)=0, \qquad \nabla \phi(x,t)\neq 0 \quad \text{for all }~(x,t)\in \Gamma_\cQ^-:=\Gamma_\cQ \setminus (x_c,t_c). 
\end{equation}
The critical point is called \emph{nondegenerate} if
\begin{equation} \label{critpoint2} 
	\det \big( \nabla^2 \phi(x_c,t_c)\big) \neq 0.
 \end{equation}
In this case, certain topological transitions can be classified via Morse theory~\cite{milnor1963morse}. The following parameter-dependent version of the Morse lemma is proved in \cite{Laurain2018}.

\smallskip \begin{lemma}\label{L:Morse} Consider a nondegenerate critical point $(x_c,t_c)$ on the zero level of $\phi$, and assume $\phi$ is $C^\infty$ in a neighborhood of $(x_c,t_c)$. Without loss of generality, assume $(x_c,t_c)=(0,0)$. Then there exists a neighborhood $\widehat X= X \times (-\delta,\delta)$ of $(0,0)$ in $\R^{d+1}$ and a map $\psi:\, \widehat X \to \R^d$ such that: $\psi(0,0)=0$, $\psi \in C^\infty(\widehat X)$, $\psi(\cdot, t)$ is a diffeomorphism from $X$ onto $\psi(X,t)$ for each $t \in (-\delta,\delta)$. Furthermore, the following normal form holds:
\begin{equation} \label{normalform}
	\phi\big(\psi(x,t),t\big)= - \sum_{i=1}^q x_i^2 + \sum_{i=q+1}^d x_i^2 + v(t), 
\end{equation}
with $0 \leq q \leq d$ and $v: (-\delta,\delta) \to \R$ such that $v(0)=0$, $v'(0)=\frac{\partial \phi}{\partial t}(0,0)$. 
\end{lemma}
\medskip 

\begin{remark} \rm The paper \cite{Laurain2018} states Lemma~\ref{L:Morse} under the assumption of $C^\infty$ regularity. Since the lemma is used here purely for classification and not in any analysis, we do not attempt to determine the minimal regularity needed for the normal form \eqref{normalform} to hold.
 \end{remark}

\medskip 
Further we are interested in critical points satisfying
\begin{equation} \label{critpoint3} 
	\frac{\partial \phi}{\partial t}(x_c,t_c) \neq 0. 
\end{equation}
This condition corresponds to an instantaneous topological transition. 

Using $\psi(x_c,t_c)=x_c$ and $\nabla \phi(x_c,t_c)=0$ we compute
\[
\begin{split}
\nabla^2\phi\big(\psi(x_c,t_c),t_c\big) &= [\nabla\psi(x_c,t_c)]^T\nabla^2\phi(x_c,t_c) \nabla\psi(x_c,t_c) + \sum_{i=1}^d\frac{\partial\phi}{\partial x_i}(x_c,t_c)\nabla^2\psi_i(x_c,t_c)\\
 &= 
 [\nabla\psi(x_c,t_c)]^T\nabla^2\phi(x_c,t_c) \nabla\psi(x_c,t_c).
 \end{split}
\]
Since $\nabla\psi(x_c,t_c)\in \R^{d\times d}$ is non-singular, the Sylvester law of inertia and \eqref{normalform} yield 
\begin{equation}\label{Inertia}
\text{Inertia}\big(\nabla^2\phi(x_c,t_c)\big) = \{d-q,q,0\}.
\end{equation}
 Recalling that the evolving domains are characterized by $\phi < 0$, thanks to the observation in \eqref{Inertia}, the normal form \eqref{normalform} and the expansion $v(t)= t\frac{\partial \phi}{\partial t}(0,0) + \mathcal{O}(t^2)$, we can classify nondegenerate critical points $(x_c,t_c)$:  Let the eigenvalues of $\nabla^2\phi(x_c,t_c)$ be
\[
\lambda_1\le\dots\le \lambda_d,\quad \lambda_i\in\text{sp}\big(\nabla^2\phi(x_c,t_c)\big). 
\]
For nondegenerate critical points all eigenvalues are non-zero and we obtain the following

\smallskip 

\noindent\texttt{Classification of nondegenerate critical points}\\
The case $d=2$:
\begin{itemize}
	\item $0<\lambda_1\le\lambda_2$,   $\frac{\partial \phi}{\partial t}(x_c,t_c) > 0$: \emph{vanishing of an island}, \\[0.2ex]
	\hspace*{1.9cm} $\frac{\partial \phi}{\partial t}(x_c,t_c) < 0$: \emph{creation of an island}.\\[-1.5ex]
	\item $\lambda_1<0<\lambda_2$,  $\frac{\partial \phi}{\partial t}(x_c,t_c) > 0$:  \emph{domain splitting}, \\[0.2ex]
	\hspace*{1.9cm} $\frac{\partial \phi}{\partial t}(x_c,t_c) < 0$: \emph{domain merging}.\\[-1.5ex]
	\item $\lambda_1\le\lambda_2<0$,  $\frac{\partial \phi}{\partial t}(x_c,t_c) > 0$:  \emph{creation of a hole}, \\[0.2ex]
	\hspace*{1.9cm} $\frac{\partial \phi}{\partial t}(x_c,t_c) < 0$: \emph{vanishing of a hole}.\\[-1.5ex]
\end{itemize}
The case $d=3$:\smallskip
\begin{itemize}
	\item $0<\lambda_1\le\lambda_2\le\lambda_3$,   $\frac{\partial \phi}{\partial t}(x_c,t_c) > 0$: \emph{vanishing of an island}, \\[0.2ex]
	\hspace*{2.8cm} $\frac{\partial \phi}{\partial t}(x_c,t_c) < 0$: \emph{creation of an island}.\\[-1.5ex]
	\item $\lambda_1<0<\lambda_2\le\lambda_3$,  $\frac{\partial \phi}{\partial t}(x_c,t_c) > 0$:  \emph{domain splitting}, \\[0.2ex]
	\hspace*{2.8cm} $\frac{\partial \phi}{\partial t}(x_c,t_c) < 0$: \emph{domain merging}.\\[-1.5ex]
	\item  $\lambda_1\le\lambda_2<0<\lambda_3$,  $\frac{\partial \phi}{\partial t}(x_c,t_c) > 0$: \emph{creation of a hole through the domain}, \\[0.2ex]
	\hspace*{2.8cm} $\frac{\partial \phi}{\partial t}(x_c,t_c) < 0$: \emph{vanishing of a hole through the domain}.\\[-1.5ex]
	\item $\lambda_1\le\lambda_2\le\lambda_3<0$,  $\frac{\partial \phi}{\partial t}(x_c,t_c) > 0$: \emph{creation of an interior void}, \\[0.2ex]
	\hspace*{2.8cm} $\frac{\partial \phi}{\partial t}(x_c,t_c) < 0$: \emph{vanishing of an interior void}.\\[-1.5ex]
\end{itemize}      
For $d=2$, a domain splitting scenario is illustrated in Figure~\ref{fig1}.

\begin{figure}[ht!] 
	\begin{center} 
		\includegraphics[width=0.45\textwidth]{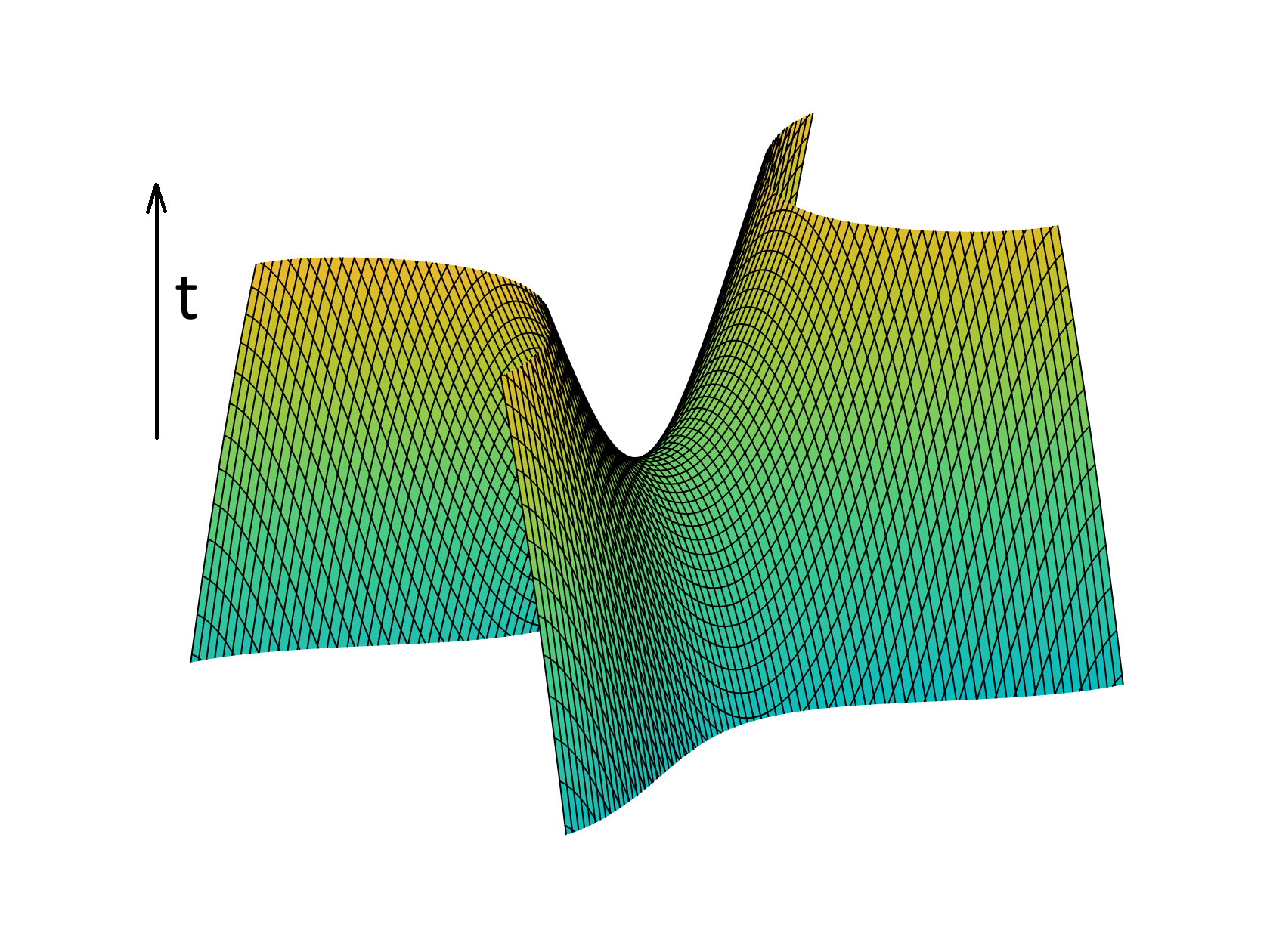} 
		\includegraphics[width=0.4\textwidth]{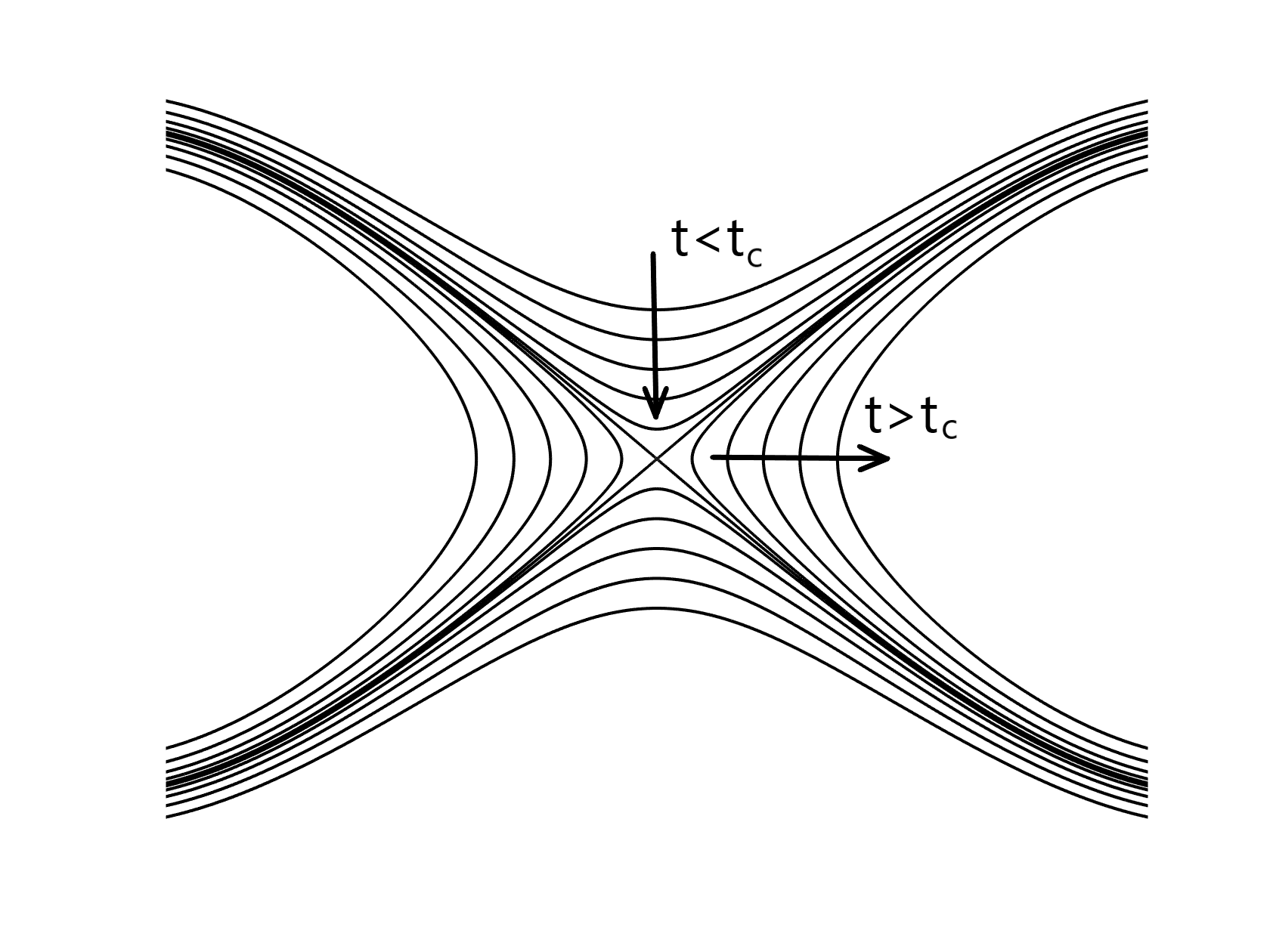} 
	\end{center} 
	\vskip-1.5em 
	\caption{$d=2$. A splitting scenario with $\lambda_1<0<\lambda_2$,  $\frac{\partial \phi}{\partial t}(x_c,t_c) > 0$. The left plot shows $\GQ$ near $(x_c,t_c)$, while the right shows snapshots of $\Gamma(t)$ near $(x_c,t_c)$.} \label{fig1} 
\end{figure}

From \eqref{critpoint3} and the expression $V_\Gamma=-\frac{\partial \phi}{\partial t} \frac{1}{|\nabla \phi|}$, we find that at an isolated critical point $(x_c,t_c)$,
\begin{equation} 
	\label{eq8} \lim_{\small \begin{matrix} (x,t) \in \GQ^-  \\ (x,t) \to (x_c,t_c) \end{matrix}} |V_\Gamma (x,t)|= \infty. 
\end{equation}

It is quite straightforward to verify that for sufficiently smooth $\phi$,  the cases classified as domain splitting, hole creation, and vanishing of an island correspond to a weakly singular point, cf. Definition~\ref{defweak}.  However, we want to be more general and check whether Assumptions~\ref{Ass1}--\ref{Ass4} can be verified for a critical point of a level set function $\phi$ which is only  $C^2$ smooth (condition~\eqref{eq:smooth}). Such results are presented in Section~\ref{subsec42}.

\subsection{The case of $\frac{\partial \phi}{\partial t}(x_c,t_c) > 0$} \label{subsec42}.
For the classification cases with  $\frac{\partial \phi}{\partial t}(x_c,t_c) > 0$ we derive some positive results.
\begin{lemma} \label{lemLevelset}
Assume \eqref{eq:smooth}  and let $(x_c,t_c)$ be an isolated critical point with $\frac{\partial \phi}{\partial t}(x_c,t_c) > 0$. Then the Assumptions~\ref{Ass1} and \ref{Ass3} are satisfied.
\end{lemma}
\begin{proof}
 Note that for $(x,t) \in \GQ^-$ we have 
\[
V_\Gamma=- \frac{\partial \phi}{\partial t}\frac{1}{|\nabla \phi|}
\]
 and thus  $V_\Gamma(x,t) < 0$ iff $\frac{\partial \phi}{\partial t}(x,t) > 0$. 
For   $ \phi \in C^2(\R^d \times [0,T])$ the condition $\frac{\partial \phi}{\partial t}(x_c,t_c) > 0$ in a critical point $(x_c,t_c)$ implies that there exists a ball $B_\delta=B_\delta\big((x_c,t_c)\big)$ with $\delta >0$ such that
\begin{equation} \label{h5} \frac{\partial \phi}{\partial t} > 0 \quad \text{on}~~\Gamma_\cQ^- \cap B_\delta.
\end{equation}
We can now use essentially the same arguments as in the proof of Lemma~\ref{Cor1}:
\[
\|[V_\Gamma]_+\|_{L^\infty(\GQ)} \leq \|V_\Gamma\|_{L^\infty(\GQ \setminus B_\delta)} = \left\|\frac{\partial \phi}{\partial t}\frac{1}{|\nabla \phi|}\right\|_{L^\infty(\GQ \setminus B_\delta)} < \infty,
\]
and thus Assumption~\ref{Ass1} is satisfied. To show that Assumption~\ref{Ass3} is satisfied, we use the proof of Lemma~\ref{lemtrace}, which is based on level set function arguments.  We split the level set evolution in two smooth parts $t \in \cI_1 \cup \cI_2$, $\cI_1:=[0,t_c)$, $\cI_2:=(t_c,T]$. For given $t \in \cI_1$ we write $\phi(x)=\phi(x,t)$, $\tilde \Gamma=\Gamma(t) \setminus B_\delta$, with $B_\delta$ as in \eqref{h5}. Then \eqref{h7} is satisfied and all the following arguments in the proof of Lemma~\ref{lemtrace} apply.\\
\end{proof}

\begin{remark} \label{RemPhi} \rm  An alternative approach for proving Lemma~\ref{lemLevelset} is to directly apply the results derived in the more general setting in Section~\ref{Sectnonsmooth}. If we show that for a level set function $\phi$ that satisfies  \eqref{eq:smooth}   an isolated critical point $(x_c,t_c)$ with $\frac{\partial \phi}{\partial t}(x_c,t_c) > 0$ is a weak singular point as in Definition~\ref{defweak}, then  Lemma~\ref{lemLevelset} follows as a corollary of the results derived in Section~\ref{Sectnonsmooth}.  However, this would require constructing suitable diffeomorphisms $\Phi_i$ from $\phi$, which may need a stronger smoothness assumption then \eqref{eq:smooth}. For this reason, we prefer the direct proof presented above. 
\end{remark}

\smallskip

Concerning Assumption~\ref{Ass4} we have  the following.

\begin{lemma} \label{LemAss4}
	Assume that $\phi$ satisfies \eqref{eq:smooth} and has one isolated critical point at which \eqref{critpoint3} holds. Then Assumption~\ref{Ass4} is satisfied.  The lateral boundary  $\GQ$ even possesses $C^2$ smoothness.
\end{lemma}
\begin{proof}
	It follows from  $\nabla_{(x,t)}\phi= (\nabla \phi, \frac{\partial \phi}{\partial t})^T$ and the implicit function theorem that $\GQ$ is $C^2$ smooth.
	Since $t_c\in(0,T)$ implies $|\nabla\phi| \ge c> 0$ on $\Gamma(0)\cup\Gamma(T)$,  $\cQ$ is Lipschitz. 
\end{proof}
\smallskip

The results above can be summarized as follows.

\smallskip
\begin{corollary} Assume \eqref{eq:smooth} holds.
For   a topological change  in the level set domain characterized by a  critical point with $\frac{\partial \phi}{\partial t}(x_c,t_c) > 0$, all the assumptions from Section~\ref{Sectdomains} are satisfied. If the critical point is  nondegenerate, these scenarios include domain splitting, cf. Figure~\ref{fig1}, the vanishing of an island, and the creation of a hole.
\end{corollary}
 
 \smallskip
\begin{remark}[Degenerate critical point]\rm
	Note that in Lemmas~\ref{lemLevelset}--\ref{LemAss4} we do not need the assumption \eqref{critpoint2} that the critical point is nondegenerate. This assumption is needed in the derivation of the classification based on Lemma~\ref{L:Morse}. We are not aware of an analogue to the classification Lemma~\ref{L:Morse} for degenerate critical points.
	
Although it does not follow from the  classification Lemma~\ref{L:Morse},
 a topological change may happen 
at an isolated \emph{degenerate} critical point.  
The following example of  a
 level set function  $\phi$ defined in a neighborhood of the degenerate critical point $x_c=(0,0)$, $t_c=0$ corresponds to  a domain splitting: 
\[
\phi(x,t) = x_2^2-|x_1|^{2p}+t\quad   \text{in}~~\cO(x_c,t_c),\quad p>1.
\]  
  See Figure~\ref{fig2} for a visualization with $p=3$. The case visually resembles the pinchoff  of a viscous drop~\cite{shi1994cascade,eggers1997nonlinear}, when the drop develops a  neck near the point of breakup.
  \end{remark}
\begin{figure}[h]
	\begin{center}
		\includegraphics[width=0.45\textwidth]{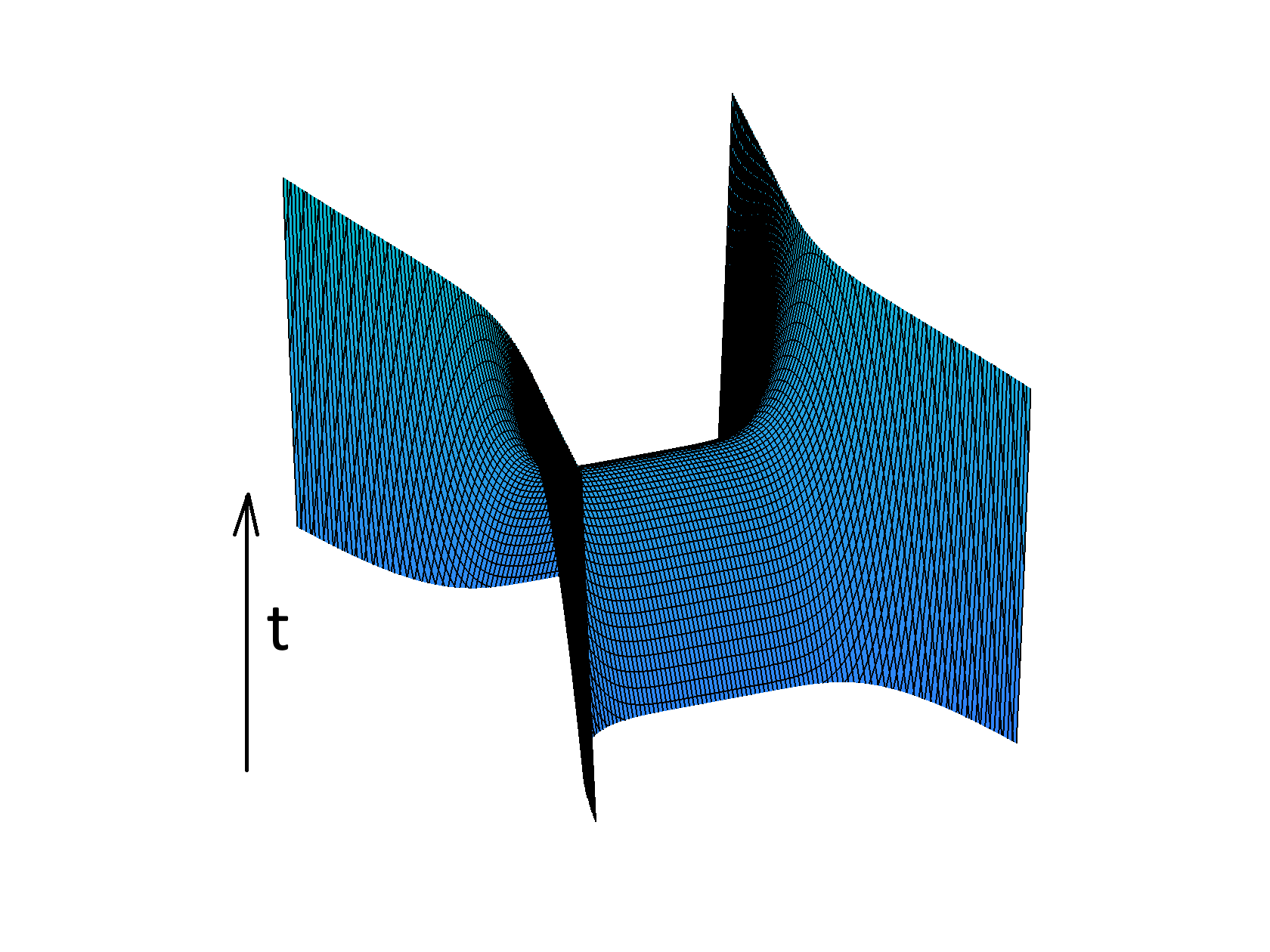}
		\includegraphics[width=0.4\textwidth]{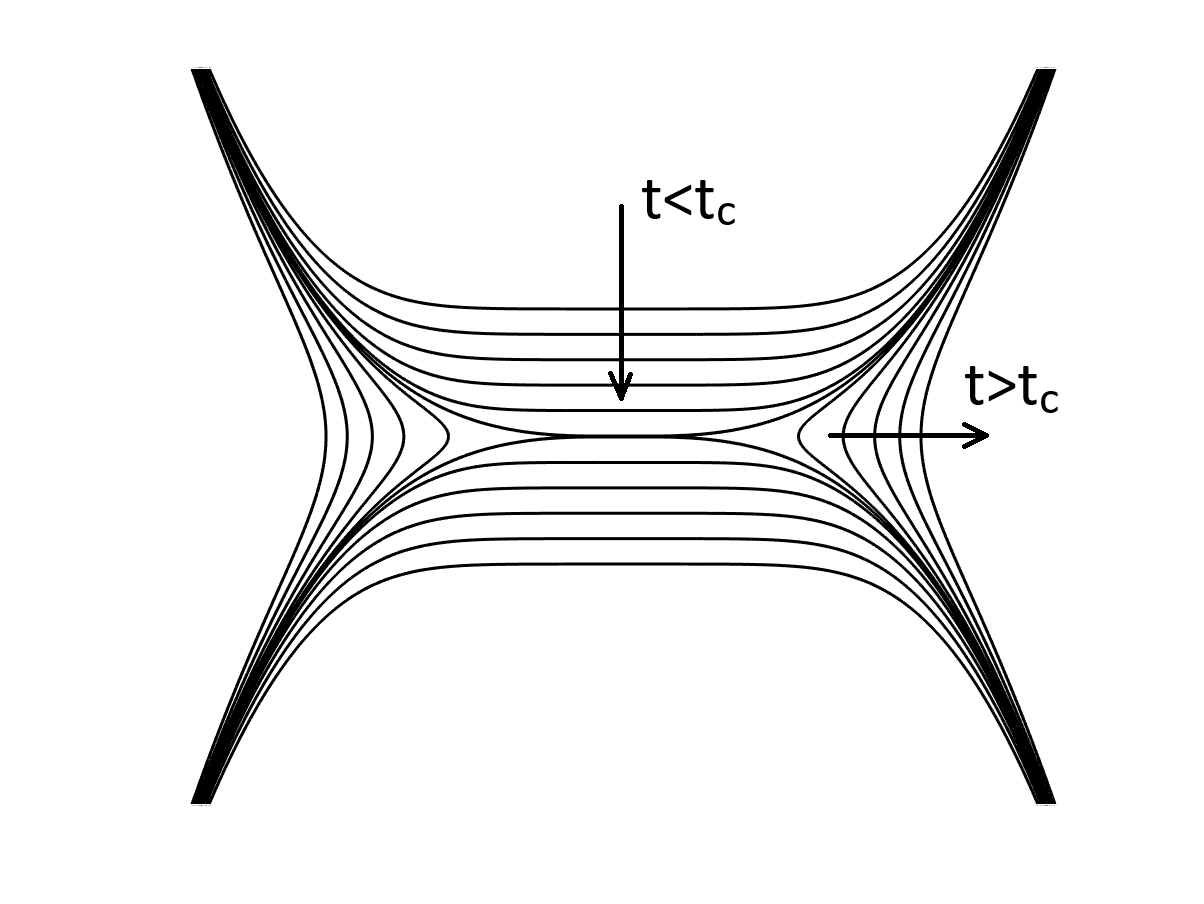}
	\end{center}
	\caption{A splitting scenario with $\phi(x,t) =-|x_1|^{6} +x_2^2+t$ in $\cO(x_c,t_c)$, corresponding to a degenerate critical point. The left plot visualizes $\GQ$, while the right plot shows snapshots of $\G(t)$. Both in a neighborhood of $(x_c,t_c)$.} \label{fig2}
\end{figure}

\subsection{A counterexample to \eqref{eq2b}} 
\rev{  
We reconsider the critical estimate \eqref{eq2b}, specifying the function space for
$u$ appearing in \eqref{eq2b}. For a fixed (sufficiently small) $\delta >0$, define
$I_\delta:=[t_c-\delta, t_c+\delta]$ and assume $\Omega^\ast \subset \R^d$ such that
$\Omega(t) \subset \Omega^\ast$ for all $t \in I_\delta$. For
$t \in I_\delta \setminus \{t_c\}$, let $C_0(t) <\infty$ be the best constant such that
\begin{equation} \label{eq2bagain}
	\frac{d}{dt}\|u\|_{\Omega(t)}^2 \leq {C_0(t)}\|u\|_{\Omega(t)}\|u\|_{H^1(\Omega(t))} 
	\quad \text{for all}~u \in H^1(\Omega^\ast).
\end{equation}
From Lemma~\ref{lemLevelset} and the estimate \eqref{eq2} we have that,
if $(x_c,t_c)$ is an isolated critical point with
$\frac{\partial \phi}{\partial t}(x_c,t_c) > 0$, then
$\sup_{t \in I_\delta} C_0(t) < \infty$ holds.

We now investigate the case of a non-degenerate critical point with
$\frac{\partial \phi}{\partial t}(x_c,t_c) < 0$. In this case,
$[V_\Gamma]_+ = V_\Gamma$ in a neighborhood of $(x_c,t_c)$, and hence
$[V_\Gamma(x,t)]_+ \to +\infty$ as $(x,t) \to (x_c,t_c)$, cf.~\eqref{eq8}. We show that
in \eqref{eq2bagain} we then have $\lim_{t \to t_c} C_0(t)=\infty$. \\[1ex]
We distinguish the different scenarios in the classification, cf.~1), 2), and 3) below.\medskip

1) \emph{Creation of an island} ($d=2,3$). Consider the level set function
$\phi(x,t)=|x|^2-t$ with corresponding domain
$\Omega(t)=\{\, x \in \R^d \,|\, |x|\leq t^\frac12\,\}$,
$0=t_c < t \leq 1$, with normal velocity $V_\Gamma=\tfrac12 |x|^{-1}$. For the constant
function $u=1$ we have
$\|u\|_{\Omega(t)}\|u\|_{H^1(\Omega(t))}=\|1\|_{\Omega(t)}^2=|\Omega(t)| \sim t^{\frac{d}2}$,
and
\[
\frac{d}{dt}\|u\|_{\Omega(t)}^2
= \int_{\Gamma(t)} V_\Gamma\cdot 1 \, ds
= \tfrac12 t^{-\frac12} |\Gamma(t)|
\sim t^{\frac{d}2-1}.
\]
Hence, for the constant from \eqref{eq2bagain} we have $C_0(t) \gtrsim t^{-1}$, and so it
blows up for $t \downarrow 0$. This blow-up corresponds to the unbounded normal velocity
for $(x,t)\to (x_c,t_c)=(0,0)$ and a blow-up of $C(t)$ in the trace inequality
$\|u\|_{\Gamma(t)}^2 \leq C(t)\|u\|_{\Omega(t)}\|u\|_{H^1(\Omega(t))}$.
\\[0.5ex]
2) \emph{Vanishing of a hole} ($d=2,3$) \emph{or of an interior void} ($d=3$).
Consider the level set function $\phi(x,t)=-|x|^2-t$ with corresponding domain
$\Omega(t)=\{\, x \in \R^d \,|\, |t|^\frac12 \leq |x|\leq \frac12\,\}$,
$-\tfrac12 \leq t \leq t_c=0$, with normal velocity $V_\Gamma=\tfrac12 |x|^{-1}$ and
$\Omega^\ast:=\Omega(0)$. For $d=2$ we take the function
$u(x)=\ln \ln(|x|^{-1})$. Note that $u$ does not depend on $t$ and we have
$u \in H^1(\Omega^\ast)$ with 
$\|u\|_{\Omega(t)}\|u\|_{H^1(\Omega(t))} \leq c$, for a suitable constant $c$
independent of $t$. On the other hand, we have
\[
\frac{d}{dt}\|u\|_{\Omega(t)}^2
= \int_{\Gamma(t)} V_\Gamma u^2 \, ds
= \tfrac12 |t|^{-\frac12} \ln^2 \ln (|t|^{-\frac12}) |\Gamma(t)|
= \pi \ln^2 \ln (|t|^{-\frac12}).
\]
Hence, we have blow-up of $C_0(t)$ for $t\uparrow 0$. It also provides a counterexample
for a vanishing hole through a three-dimensional domain.
\\[0.5ex]
For a vanishing interior void ($d=3$) we were not able to find a fixed
function $u \in H^1(\Omega^\ast)$ for which $C_0(t)$ in \eqref{eq2bagain} blows up. We
can, however, construct a parameterized subset
$\{u_\epsilon\}_{\epsilon >0} \subset H^1(\Omega^\ast)$ for which blow-up occurs. Define,
for $0 < \epsilon < \tfrac12$,
$u_\epsilon(x):=|x|^{-1}$ if $|x| \geq \epsilon$ and $u_\epsilon= \epsilon^{-1}$ if
$|x| \leq \epsilon$. One verifies
$\|u_\epsilon\|_{\Omega(t)}\leq \|u_\epsilon\|_{\Omega^\ast} \leq 1$,
$\|\nabla u_\epsilon\|_{\Omega(t)} \leq \|\nabla u_\epsilon\|_{\Omega^\ast}
\lesssim \epsilon^{-\frac12}$. We thus have
$\|u_\epsilon\|_{\Omega(t)}\|u_\epsilon\|_{H^1(\Omega(t))} \leq c\epsilon^{-\frac12}$
for a suitable constant $c$, independent of $t$ and $\epsilon$. At the same time, for
$|t| \leq \epsilon^2$ we have
\[
\frac{d}{dt}\|u_\epsilon\|_{\Omega(t)}^2
= \int_{\Gamma(t)} V_\Gamma u_\epsilon^2 \, ds
= 2 \pi |t|^\frac12\epsilon^{-2}
=:C(\epsilon,t).
\]
We see that $\lim_{\epsilon \downarrow 0} \epsilon^\frac12 C(\epsilon, \epsilon^2)=\infty$,
and thus we have blow-up of $C_0(t)$ in \eqref{eq2bagain}.
\\[0.5ex]
3) \emph{Domain merging} ($d=2$). Consider the level set function
$\phi(x,t)=-x_1^2 +x_2^2-t$ with corresponding domain
$\Omega(t)=\{\, x \in \R^2 \,|\,  \sqrt{x_2^2-t} \leq |x_1| \leq 1 \,\}$,
$-\tfrac12<t<t_c=0$, with normal velocity $V_\Gamma=\tfrac12 |x|^{-1}$. Note $\Omega(t) \subset \Omega^\ast:=[-1,1]^2$. The part of
$\partial \Omega(t)$ that is evolving is denoted by
$\Gamma(t):=\{\, x \in \partial \Omega(t)\,|\, x_1 \notin \{-1,1\}\,\}$. We take the
constant function $u=1$. For this we have
$\|u\|_{\Omega(t)}\|u\|_{H^1(\Omega(t))}=\|1\|_{\Omega(t)}^2 \leq 4 $. The
evolving boundary part $\Gamma(t)$ can be parameterized by
$G(x_2)=(\pm \sqrt{x_2^2-t},x_2)$, $x_2 \in (-\sqrt{1+t},\sqrt{1+t})$, with $|G'(x_2)| \geq 1$. Thus we
obtain
\[
\begin{split}
	\frac{d}{dt}\|u\|_{\Omega(t)}^2
	& = \int_{\Gamma(t)} V_\Gamma\cdot 1 \, ds
	=2 \int_{-\sqrt{1+t}}^{\sqrt{1+t}} \tfrac12 |G(x_2)|^{-1} |G'(x_2)|\, dx_2
	\geq \int_{-\frac12 \sqrt{2}}^{\frac12 \sqrt{2}} \frac{1}{\sqrt{2x_2^2 +|t|}} \, dx_2 \\
	& \geq \frac{1}{\sqrt{2}}\int_{-\frac12 \sqrt{2}}^{\frac12 \sqrt{2}} \frac{1}{\sqrt{x_2^2 +|t|}} \, dx_2
	= \frac{1}{\sqrt{2}} \ln \left( \frac{\left(\frac12 \sqrt{2}+\sqrt{\frac12+|t|}\right)^2}{|t|}\right)
	\sim c \, \ln \left(\frac{\sqrt{2}}{|t|}\right)
	\quad \text{for}~|t| \downarrow 0.
\end{split}
\]
Hence, we have a blow-up of $C_0(t)$ for $|t| \downarrow 0$.
\\[0.3ex]
\emph{Domain merging} ($d=3$). Consider the level set function
$\phi(x,t)=-x_1^2 +x_2^2+ x_3^2 -t$ with corresponding domain
$\Omega(t)=\{\, x \in \R^3 \,|\,  \sqrt{x_2^2+x_3^2-t} \leq |x_1| \leq 1 \,\}$,
$-\tfrac12<t<t_c=0$, with normal velocity $V_\Gamma=\tfrac12 |x|^{-1}$. Note $\Omega(t) \subset \Omega^\ast:=[-1,1]^3$. The part of
$\partial \Omega(t)$ that is evolving is denoted by
$\Gamma(t):=\{\, x \in \partial \Omega(t)\,|\, x_1 \notin \{-1,1\}\,\}$, which is the surface of revolution of $g(x_1):=\sqrt{x_1^2-|t|}$ around the $x_1$-axis for $x_1 \in [-1,-|t|^\frac12] \cup [|t|^\frac12, 1]$.  Now, for providing a counterexample, instead of a fixed 
$u \in H^1(\Omega^\ast)$ we again consider the above
parameterized family of functions $\{u_\epsilon\}_{\epsilon >0} \subset H^1(\Omega^\ast)$, satisfying
$\|u_\epsilon\|_{\Omega(t)}\|u_\epsilon\|_{H^1(\Omega(t))} \leq c\epsilon^{-\frac12}$. 
 For $|t| \leq \tfrac14\epsilon^2$  the surface of revolution restricted to $x_1 \in [-\epsilon,-|t|^\frac12] \cup [|t|^\frac12, \epsilon]$ is denoted by $\Gamma_\epsilon(t)$. For $x \in \Gamma_\epsilon(t)$ we have $|x|=(2x_1^2-|t|)^\frac12 \leq \sqrt{2}\epsilon$. Using this and the definition of $u_\epsilon$ we obtain, with suitable constants $c_i > 0$ independent of $\epsilon$ and $t$: 
\[ \begin{split}
\frac{d}{dt}\|u_\epsilon\|_{\Omega(t)}^2
& = \int_{\Gamma(t)} V_\Gamma u_\epsilon^2 \, ds
\ge \int_{\Gamma_\epsilon(t)} V_\Gamma u_\epsilon^2 \, ds
\ge c_1  \epsilon^{-3} |\Gamma_\epsilon(t)| \\ &  = c_2  \epsilon^{-3} \int_{|t|^\frac12}^\epsilon g(x_1) \left(1+g'(x_1)^2\right)^\frac12 dx_1 
 =  c_2 \epsilon^{-3} \int_{|t|^\frac12}^\epsilon \left(2 x_1^2 -|t|\right)^\frac12 dx_1 \geq c_3 \epsilon^{-1}.
\end{split} \]
So, for $|t| \leq \tfrac14\epsilon^2$ and $\epsilon \downarrow 0$ we have blow-up of $C_0(t)$ in
\eqref{eq2bagain}.

Thus, in the case of topological transitions of level set domains at a non-degenerate
critical point $(x_c,t_c)$ with
$\frac{\partial \phi}{\partial t}(x_c,t_c) < 0$, we lack suitable control over the
variation of $\|u\|_{\Omega(t)}^2$. Corresponding topological transitions include
domain merging (the time-reversed case of the one in Fig.~\ref{fig1}), the creation of
an island, and the vanishing of a hole and an internal void.

We emphasize that the arguments in this section do not exclude generic cases of domain
merging, island creation, etc., from the analysis in Section~\ref{Sectnonsmooth}—only
those involving non-degenerate critical points of smooth level set functions.

}

\subsection{Examples of non-smooth transition} \label{secnonsmooth} One can think of numerous scenarios involving non-smooth topological transitions, where by \textit{non-smooth} we mean that the lateral boundary of the space-time domain $\cQ$ is not a $C^2$ manifold.

For example, consider  two colliding balls that `preserve their shape' after the collision. For $d=2$ a specific example is as follows. Let $\phi_c(x)=(x_1-c)^2+x_2^2-1$ be a level set function for a unit ball with  center $(c,0)$. For $t \in (-1,1)$ we define  
\[
\phi(x,t) = \min\{\phi_{t-1}(x), \phi_{1 - t}(x)\}, \quad x \in \R^2.
\]  
For $t<0$ the domain $\Omega(t)$ is formed by the interior of two separated balls, for $t=0$ they touch and for $t>0$ the domain
$\Omega(t)$ is formed by the interior of the two overlapping balls. For $t >0$ there are two points $x_c$ where the boundaries of the two balls intersect. Due to the $\min$ operation the boundary of $\Omega(t)$ has only Lipschitz smoothness at these intersection points $x_c$. Hence, $\Omega(t)$ is only Lipschitz for $t > t_c=0$. This does not match the setting described in Section~\ref{Sectdomains} in which we assume $C^2$ smoothness of $\Omega(t)$ before and after the critical time $t_c$.

An example of a non-smooth transition that does fit in the framework of Section~\ref{Sectdomains} is the following.
Consider the level set function $\phi$:
\begin{equation} \label{aux410}
	\phi (x,t)=  
	\begin{cases}  
		|x_1| - x_2^2 + t, & \text{for } t < 0, \\  
		x_1^2 - x_2^4 + t^p, & \text{for } t > 0, \quad p \geq 1  
	\end{cases}  
	\qquad \text{in} \quad \cO(x_c, t_c), \quad (x_c,t_c)=(0,0).
\end{equation}
In this case the interface becomes smooth instantaneously after the domains touch, cf. Figure~\ref{fig3} for a visualization. One easily checks that the space--time domains  $\cQ_i$, $i=1,2$, before and after $t=t_c$  are $C^2$ smooth. At $t=t_c$, however, we only have Lipschitz smoothness.
\begin{figure}[h]
	\begin{center}
		\includegraphics[width=0.45\textwidth]{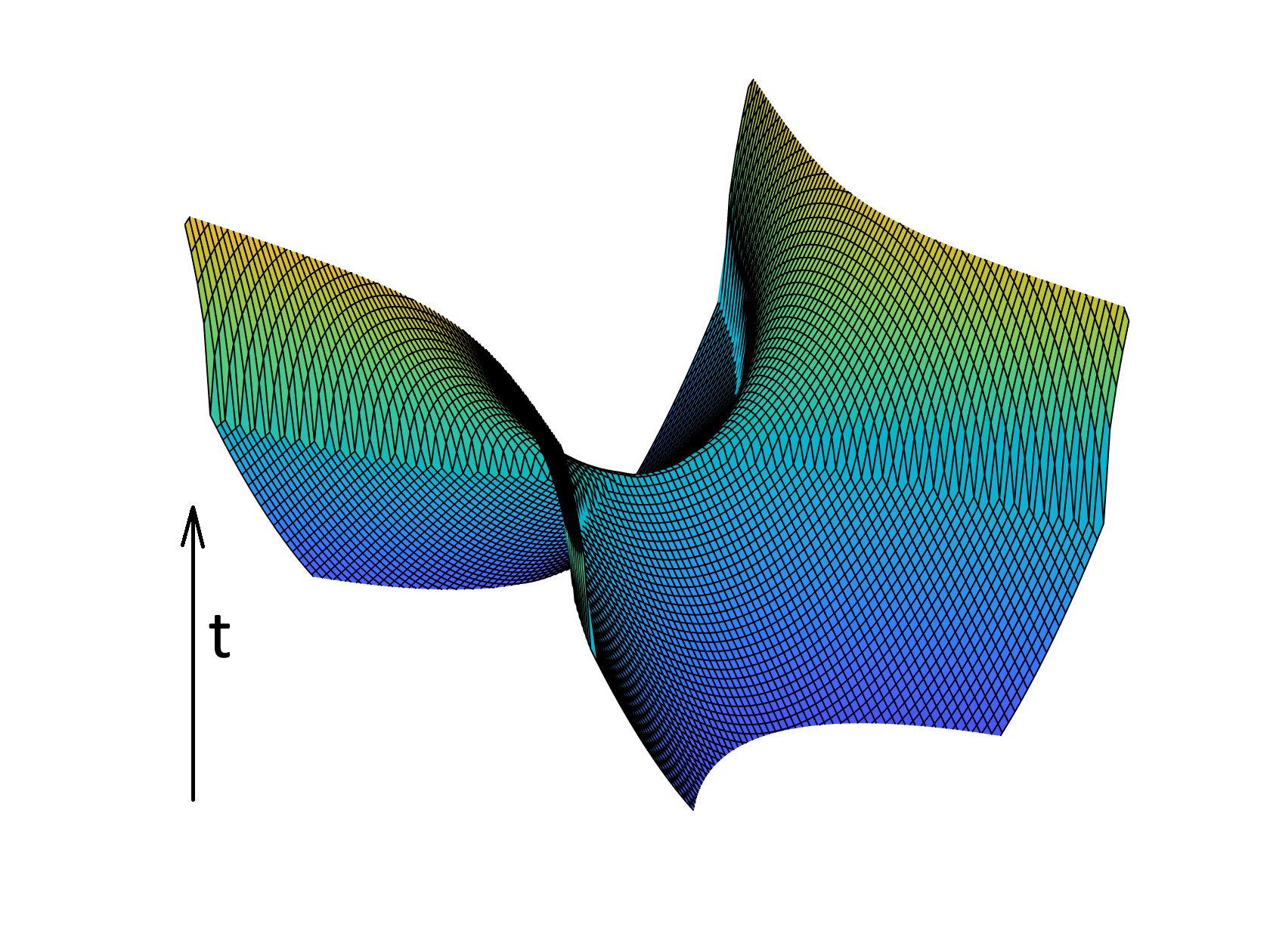}
		\includegraphics[width=0.4\textwidth]{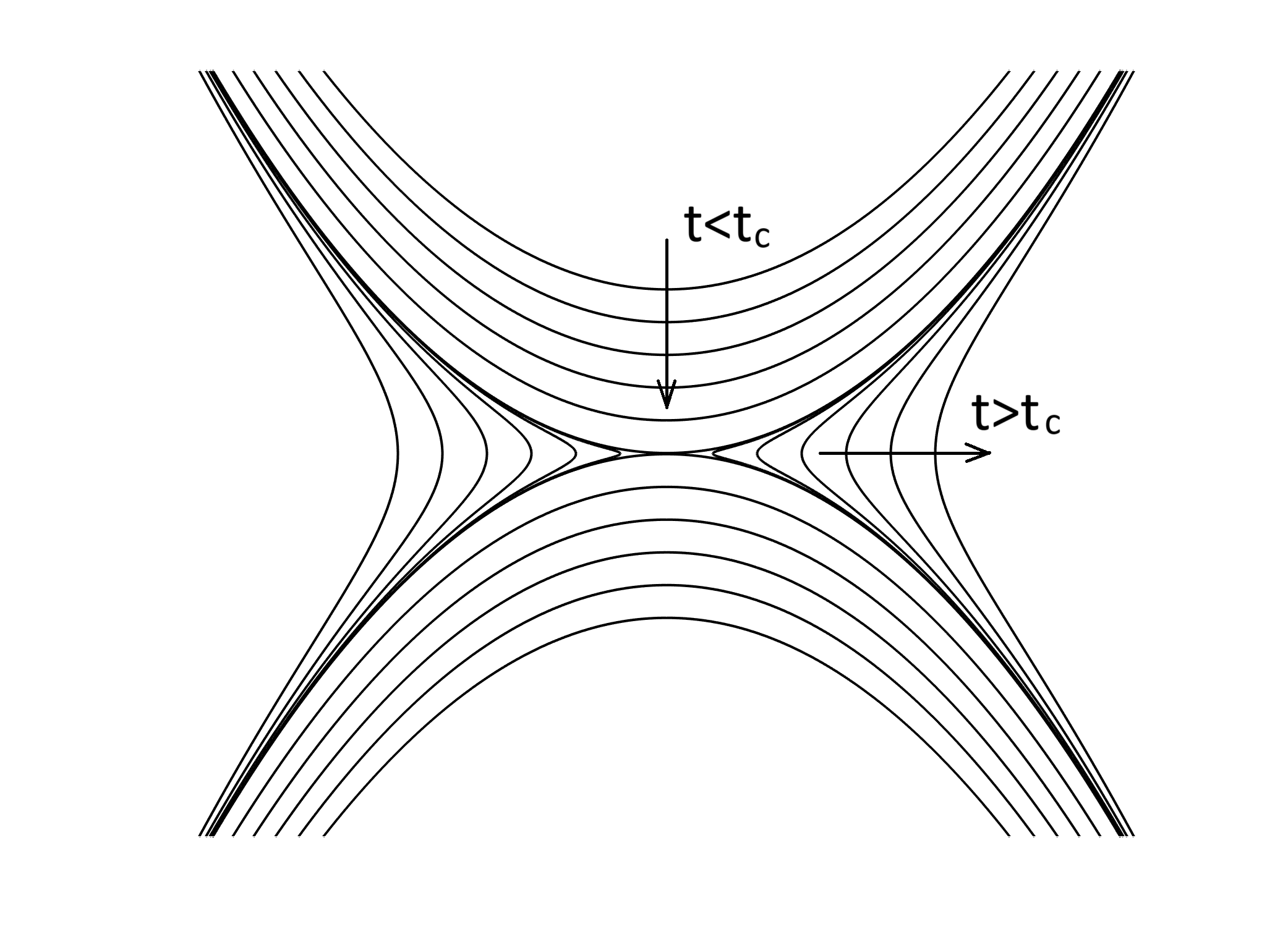}
	\end{center}
	\vskip-1.5em
	\caption{Domain merging with a `non-smooth' transition. The left plot visualizes $\GQ$, while the right plot shows snapshots of $\G(t)$, both in a neighborhood of $(x_c, t_c)$.}  
	\label{fig3}
\end{figure}

It is not difficult to check that $[V_\G]_+$ remains uniformly bounded for $t < t_c$ and $t > t_c$ if $p \geq 4$. Thus, Assumption~\ref{Ass1} holds.
However, it remains unclear whether the uniform trace inequality, as stated in Assumption~\ref{Ass3}, holds in this case.

If the time is reversed in \eqref{aux410}, then we have the domain splitting with a non-smooth transition. The transition point is  weakly singular and so both Assumptions~\ref{Ass1} and~\ref{Ass3} hold.

\subsection{Lipschitz condition for $\cQ$}
If the evolving domain is characterized by a level set function $\phi$ that is $C^1$ smooth on the space time domain and has a isolated critical point $(x_c,t_c)$ for which $\frac{\partial \phi}{\partial t}(x_c,t_c) \neq 0$ holds, then Assumption~\ref{Ass4} is satisfied, cf. Lemma~\ref{LemAss4}. For less regular $\phi$, the Lipschitz property of $\cQ$ is not guaranteed even if $\Omega(t)$ is uniformly Lipschitz in time as a spatial domain.

	\begin{figure}[h]
	\begin{center}
		\includegraphics[width=0.45\textwidth]{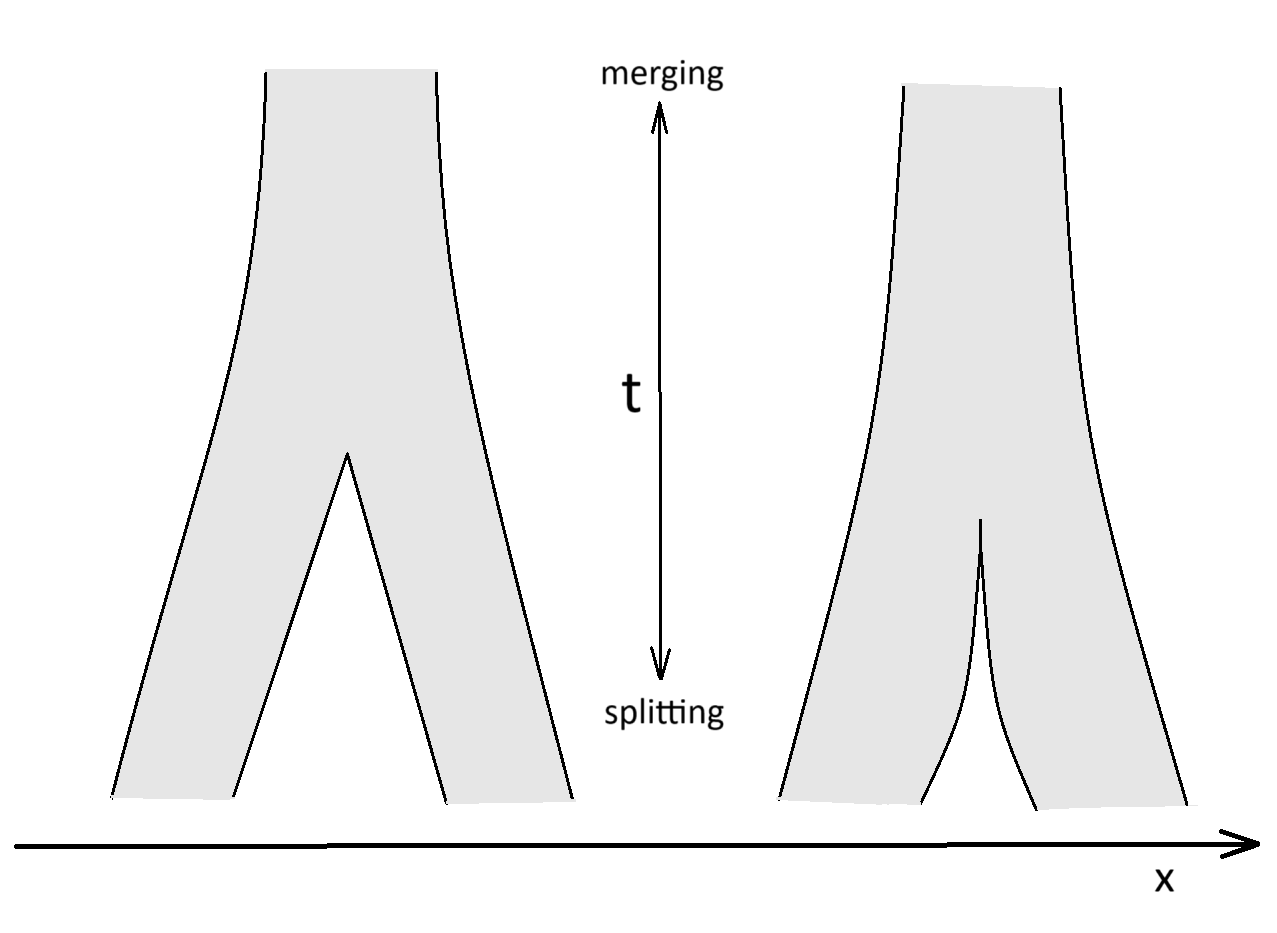}
	\end{center}
	\caption{One-dimensional domains merging/splitting with Assumption~\ref{Ass4} true (left plot) and false (right plot). \label{fig4}} 
\end{figure}

To illustrate this, we consider a simple example of merging or splitting of 1D domains, i.e., intervals in $\mathbb{R}$, with a level set function that is not $C^1$ smooth. Define
\[
  \phi(x,t)=-|x|^\alpha \pm t, \quad 0<\alpha\leq 1,\quad x \in \R,\quad t\in(-1,1).
\]
The singular point is $(x_c,t_c)=(0,0)$ and the '$+$'/'$-$' corresponds to splitting/merging. A simple computation yields that for all $\alpha$ we have $[V_\Gamma]_+=0$ in the case of splitting and $[V_\Gamma]_+\le \alpha^{-1}$ in the case of merging. Furthermore, in the splitting case we have that $\Gamma_+(t)$ is empty and in the case of merging it is a  boundary, consisting of two points,  and thus Assumption~\ref{Ass3} is trivially satisfied. However, Assumption~\ref{Ass4} is fulfilled only for $\alpha=1$, whereas for $\alpha <1$ the space--time domain is not Lipschitz, cf. Figure~\ref{fig4}.


\section{Narrow band estimate}\label{sec_key}
In this section we establish a key technical result that extends \eqref{eq2b} to any subinterval of $[0,T]$ and formulates the ``control of variation'' of $\|u\|_{\Omega(t)}^2$ in a form suitable for the stability analysis of a finite element method in Section~\ref{s_stab}. In the derivation of this result we use the Assumptions~\ref{Ass0},~\ref{Ass1} and \ref{Ass3}. 
 
Fix some $t_0\in[0,T)$ and $\Delta t>0$ such that  $t_0+\Delta t\in(0,T]$.  Consider \rev{an arbitrary} domain $\cO_{t_0}\subset \R^d$ such that 
$\Omega(t)\subset\cO_{t_0}$ for all $t\in [t_0,t_0+\Delta t]$.

\smallskip
\begin{lemma}\label{L:nb} Suppose that the evolution of $\Omega(t)$ satisfies Assumptions~\ref{Ass0}--\ref{Ass3}. For $u\in H^1(\cO_{t_0})$ and any $\eps>0$ it holds
	\begin{equation} \label{est:nb}
		\begin{split}
\|u\|_{\Omega(t_0+\Delta t)}^2- \|u\|_{\Omega(t_0)}^2 & \le  C\,\Big(\int_{t_0}^{t_0+\Delta t} \sup_{\G(t)}\,[V_\G]_+\,dt \Big) \left( (1+\eps^{-1})\|u\|^2_{\cO_{t_0}} + \eps\|\nabla u\|^2_{\cO_{t_0}} \right) \\
			& \le C\,\Delta t \left( (1+\eps^{-1})\|u\|^2_{\cO_{t_0}} + \eps\|\nabla u\|^2_{\cO_{t_0}} \right),
		\end{split}
	\end{equation}
where $C$ is independent of $\Delta t$, $t_0$ and $u$.
\end{lemma}
\begin{proof} 
	Assume  $t_c\notin(t_0,t_0+\Delta t)$. Without loss of generality
we may let $t_0=0$. Recalling that $V_\Gamma$ is the normal velocity of $\G(t)$ and employing the Reynolds transport theorem,
we have 
\begin{equation}\label{eq:aux107}
\|u\|_{\Omega(t_0+\Delta t)}^2- \|u\|_{\Omega(t_0)}^2 = \int_{0}^{\Delta t}\frac{d}{dt}\int_{\Omega(t)} u^2\,dx = \int_{0}^{\Delta t}\int_{\G(t)} V_\G u^2\,ds\,dt.
\end{equation}
This and Assumption~\ref{Ass1} imply  the estimate
\begin{equation}\label{eq:aux82b}
		\|u\|_{\Omega(t_0+\Delta t)}^2- \|u\|_{\Omega(t_0)}^2 \le \int_{0}^{\Delta t}\int_{\Gp(t)} [V_\G]_+ u^2\,ds\,dt
	\le	\int_{0}^{\Delta t} \sup_{\G(t)}[V_\G]_+\,dt \sup_{t\in[0,\Delta t]}\,\int_{\Gp(t)} u^2\,ds.
\end{equation}
Thanks to the Assumption~\ref{Ass3} for all $t\in[0,\Delta t]$ we have
\begin{equation}\label{aux111}
\begin{split}
\int_{\Gp(t)} u^2\,ds&\le   C_{\Gamma}(\Omega(t)) \| u \|_{L^2(\Omega(t))} (\| \nabla u \|_{L^2(\Omega(t))} +  \| u \|_{L^2(\Omega(t))})\\
&\leq C\big( \frac1{2\eps} \|u\|_{L^2(\Omega(t))}^{2} + \frac\eps{2}\|\nabla u\|_{L^{2}(\Omega(t))}^{2} + \|u\|_{L^2(\Omega(t))}^{2} \big)\quad\forall\eps>0.
\end{split}
\end{equation} 
  Using this in \eqref{eq:aux82b}  proves the first inequality in \eqref{est:nb} if $t_c\notin(t_0,t_0+\Delta t)$.
  The second inequality in \eqref{est:nb} directly follows from the first one using Assumption~\ref{Ass1}.

If $t_c\in(t_0,t_0+\Delta t)$ we proceed by splitting 
\begin{equation}\label{aux162}
	\|u\|_{\Omega(t_0+\Delta t)}^2- \|u\|_{\Omega(t_0)}^2\le\left(\|u\|_{\Omega(t_0+\Delta t)}^2- \|u\|_{\Omega(t_c)}^2\right)+\left(\|u\|_{\Omega(t_c)}^2- \|u\|_{\Omega(t_0)}^2\right).
\end{equation} 
Since $u^2\in L^q(\cO_{t_0})$, $q>1$,  Assumption~\ref{Ass0} implies  $\|u\|_{\Omega(t_c)}^2=\lim_{\delta\to\pm0} \|u\|_{\Omega(t_c+\delta)}^2$.
Therefore to estimate the terms on the right-hand side of \eqref{aux162} we apply the arguments above to estimate 
$\|u\|_{\Omega(t_c-\delta)}^2- \|u\|_{\Omega(t_0)}^2$ and $\|u\|_{\Omega(t_0+\Delta t)}^2- \|u\|_{\Omega(t_c+\delta)}^2$ for any sufficiently small $\delta>0$ and pass to the limit. \\
\end{proof}

\section{Model problem and finite element method}\label{s:model}
In this section, we introduce a model PDE to which we will apply a finite element discretization method.
We assume  that an evolving domain $\Omega(t)$, $t \in [0,T]$, as specified in Section~\ref{Sectdomains} is given and consider the  heat equation:
\begin{equation} \label{transport}
	  \begin{split}
				\frac{\partial u}{\partial t}  - \Delta u&=0\quad\text{on}~~\Omega(t), ~~t\in (0,T],
		  \\
		 \nabla u \cdot \bn &= 0 \quad \text{on}~~\partial \Omega(t),~~t\in (0,T],
		\end{split}
\end{equation}
with initial condition $u(x,0)=u_0(x)$ for $x \in \Omega(0)$. 
As far as we know, well-posedness of this  equation on an evolving domain with a topological change has not been studied in the literature so far. Here we \emph{assume} that there is a  solution $u$  to \eqref{transport} with regularity 
\begin{equation}\label{eq:reg}
	u\in W^{2,\infty}(\cQ)\cap H^{k}(\cQ),
\end{equation}
for some integer $k \geq 2$.
Clearly, if the evolving domain  corresponds to two separated domains that merge, such a regularity assumption is not reasonable. On the other hand, for domain splitting with a smooth space--time domain, cf. Lemma~\ref{LemAss4}, we consider this to be a plausible assumption. 

\begin{remark} \label{RemNew} \rm
To analyze the well-posedness of \eqref{transport}, a standard approach is to first show \textit{a priori} estimates satisfied by any  sufficiently regular solution. 
Assume $u\,:\,\Omega(t)\to\R$ is such a solution and $t\neq t_c$.  The Reynolds transport theorem and \eqref{transport} yield: 
\begin{equation} \label{eq:1}
	\begin{split}
	\frac{d}{dt}\|u\|_{\Omega(t)}^2& =\frac{d}{dt} \int_{\Omega(t)} u^2 \, dx = 2\int_{\Omega(t)} u\partial_t u+ \int_{\Gamma(t)} V_\Gamma u^2 \, ds = 2\int_{\Omega(t)} u\Delta u+ \int_{\Gamma(t)} V_\Gamma u^2 \, ds\\
	& = -2\|\nabla u\|_{\Omega(t)}^2+ \int_{\Gamma(t)} V_\Gamma u^2 \, ds.
	\end{split}
\end{equation}
In general the last term in \eqref{eq:1} is sign indefinite and we can handle in with the help of Assumptions~\ref{Ass1}--\ref{Ass3}, with the same  arguments as used in \eqref{eq2}--\eqref{eq2b}: 
\begin{equation*}
		\begin{split}
	\frac{d}{dt}\|u\|_{\Omega(t)}^2 &= -2\|\nabla u\|_{\Omega(t)}^2+ \int_{\Gamma(t)} V_\Gamma u^2 \, ds 
	 \\
	& \leq -2\|\nabla u\|_{\Omega(t)}^2 + C \Big((1+\frac{1}{2\eps})\|u\|_{\Omega(t)}^2 + 2\eps \|\nabla u\|_{\Omega(t)}^2\Big)
	\\
	&\leq  -\|\nabla u\|_{\Omega(t)}^2 + \tilde C\|u\|_{\Omega(t)}^2,\quad t\in(0,T),~ t\neq t_c
	\end{split}
\end{equation*}
with a suitable constant $\tilde C$ that is independent of $u$. 
Now the first \textit{a priori} estimate
\begin{equation} \label{eq:apr1}
	\sup_{t\in[0,T]}\|u\|_{\Omega(t)} +	\|\nabla u\|_{\cQ}\le C\|u\|_{\Omega(0)}
\end{equation} 
follows by applying the Gronwall inequality and Assumption~\ref{Ass0} to extend it over the critical time. We see that Assumptions~\ref{Ass0}--\ref{Ass3}
may also be critical for studying the well-posedness of the model problem~\eqref{transport}. 
\end{remark}

\begin{remark}\rm 
 A more physically motivated equation can be   
\begin{equation} \label{transport2}
\frac{\partial u}{\partial t} + \bw\cdot\nabla u -  \Delta u=0\quad\text{on}~~\Omega(t), ~~t\in (0,T],
\end{equation}
with an ALE velocity $\bw$ built from the domain diffeomorphisms $\Psi_i$: 
\begin{equation}\label{Lagrange}
\bw(t,\Psi_i(t,\by))=\frac{\partial \Psi_i(t,\by)}{\partial t},\quad t\in I_i,~\by\in \Omega_i^0.
\end{equation}

However, error analysis for a finite element method applied to \eqref{transport2} is more difficult to analyze, due to $\|\bw\|_{L^\infty(\Omega(t))} \to \infty$  for $t\to t_c$. So we start with the simpler problem \eqref{transport}.
\end{remark}

\subsection{Finite element method}
For the discretization of the heat equation \eqref{transport}, we use an Eulerian unfitted finite element method (FEM),  introduced in~\cite{lehrenfeld2019eulerian}. In that paper an error analysis of the method is presented for the case of \emph{smoothly} evolving domains.
We begin by describing the method.

We assume that for all times $t \in [0, T]$, the physical domain $\Omega(t)$ is embedded in a fixed bulk domain $\wOm$, and we consider a family of consistent subdivisions of $\wOm$ into quasi-uniform triangulations $\{\T_h\}_{h>0}$, consisting of simplices with characteristic mesh size $h$. This defines a family of background, time-independent triangulations.

For the spatial discretization, we use a time-independent finite element space,
\begin{equation} \label{eq:Vh} 
	V_h := \{v \in C(\wOm) \, :\, v|_T \in P_m(T),\ \forall T \in \mathcal{T}_h\},\quad m \ge 1, 
	\end{equation}
where $P_m(T)$ denotes the space of polynomials of degree at most $m$ on $T$.

The discretization method combines a standard time-stepping procedure with an implicit extension of the finite element solution, which allows time-stepping on an unfitted triangulation. The time interval $[0, T]$ is discretized using time steps $t_n = n \Delta t$, for $0 \leq n \leq N$. An extension is constructed on a domain formed by enlarging $\Omega^n := \Omega(t_n)$ with a few additional layers of simplices. This extended domain is denoted by $\Odt{n} \subset \wOm$. The width of the additional layer is determined by a parameter $\delta > 0$: 
\begin{equation} \label{def5} 
	\overline{\Odt{n}} := \bigcup_{T \in \Td{n}} \overline{T},\quad \text{where}~ \Td{n} := \{ T \in \T_h : \dist(x, \Omega^n) \leq \delta \text{ for some } x \in T \}. 
\end{equation}

The parameter $\delta$ may vary between time steps and is chosen such that the following key condition holds: \begin{equation} \label{cond1} \delta = \delta(n) \text{ is sufficiently large such that } \rev{\Omega(t)} \subset \Odt{n-1},~\rev{\text{for}~t\in[t_{n-1},t_n]} \quad n = 1, \dots, N. 
\end{equation}

This property ensures that in a time step from $t_{n-1}$ to $t_n$, an approximation $u_h^{n-1} \approx u(\cdot, t_{n-1})$ can be used in the variational formulation at time $t_n$, which involves integration over the domain $\Omega^n$, cf. \eqref{e:unfFEM1} below.
Thanks to Assumption~\ref{Ass1}, it is sufficient to choose $\delta \ge V_{\max}^+ \Delta t$ to ensure \eqref{cond1}. Thus, we may take \begin{equation} \label{delta} \delta \simeq \Delta t. \end{equation}

The triangulation $\Td{n}$ defined in \eqref{def5} constitutes the \emph{active mesh} at time step $n$. On this active mesh, we define the finite element space 
\begin{equation} \label{defV} 
V_h^n := \{v \in C(\Odt{n})\,:\, v|_T \in P_m(T),\ \forall T \in \Td{n} \},\quad m \ge 1.
 \end{equation}
  Note that this space is obtained by restricting functions from the time-independent bulk space $V_h$ to all simplices in $\Td{n}$.

The finite element \rev{in space and backward Euler in time} discretization of \eqref{transport} reads as follows:
Given $u_h^0 \rev{= \mathcal{I}u^0} \in V_h^0$, find $u_h^n \in V_h^n$, for $n = 1, \dots, N$, such that
\begin{equation} \label{e:unfFEM1} \int_{\Omega^n} \frac{u_h^n - u_h^{n-1}}{\Delta t} v_h dx + a^n(u_h^n, v_h) + \gamma_s \, s_h^n(u_h^n, v_h) = 0 \quad \text{for all } v_h \in V_h^n, 
\end{equation} 
with 
\begin{equation} \label{e:anh} a^n(u_h, v_h) := \int_{\Omega^n} \nabla u_h \cdot \nabla v_h  dx, 
\end{equation}
and where $s_h^n(\cdot, \cdot)$ is a stabilization bilinear form with corresponding stability parameter $\gamma_s > 0$.

The bilinear form $s_h^n(\cdot, \cdot)$ serves two purposes: it mitigates instabilities caused by irregular intersections of the mesh with the domain boundary, and it facilitates the \emph{implicit extension} of the finite element solution to $\Odt{n}$. Several suitable choices for $s_h^n$ are available in the literature. In this work, we use the “direct” version of the ghost penalty method proposed in \cite{preussmaster,lehrenfeld2019eulerian}, which is described in Appendix~\ref{A1}. The choice of the stability parameter $\gamma_s$ is based on stability analysis and is given in \eqref{choicegamma}.

\begin{remark} \label{remimplementation} \rm Note that the finite element discretization in \eqref{e:unfFEM1} is an \emph{un}fitted method. The triangulation $\mathcal{T}_h$ is not aligned with the domains $\Omega^n$, for $n = 1, \ldots, N$. Consequently, an implementation problem arises concerning numerical integration over cut elements.	Sufficiently accurate quadrature can be achieved by representing the evolving domain with a level set function and applying the parametric finite element technique introduced in \cite{lehrenfeld2015cmame}. We do not discuss these aspects further here and instead refer the reader to \cite{lehrenfeld2015cmame,lehrenfeld2019eulerian}.
	In the following analysis, the effect of quadrature errors is not taken into account.
\end{remark}

\section{Stability and error analysis} \label{SecError}
In this section, we present an error analysis of the finite element method~\eqref{e:unfFEM1}. A key ingredient is the estimate derived in Lemma~\ref{L:nb}, which allows us to obtain a stability estimate that matches the one for smoothly evolving domains without topological changes; see Theorem~\ref{s_stab} below. This stability result can be combined with standard consistency error estimates, leading to the discretization error bound stated in Theorem~\ref{Th2}.
The analysis closely follows the approach in \cite{lehrenfeld2019eulerian} for the case of smoothly evolving domains. Although the arguments used in the consistency and error analyses in Sections~\ref{sec:consist} and~\ref{sec:aprioriest} are standard, we include them here for completeness.

\subsection{Preliminaries} \label{s:prelim}
In the remainder we write $a \lesssim b$ to state that the inequality $a \leq c \, b$ holds for quantities $a$ and $b$, with a constant $c$, which is independent of the discretization 
parameters $h$, $\Delta t$, the time instance $t_n$, and the position of $\Omega^n$ in the background mesh.
Let $\mathcal{I}$ be the Lagrange interpolation operator \rev{defined on the extended computational domain}.  We will need the following consistency result for the stabilization term~\cite{lehrenfeld2019eulerian}: 
For $w \in H^{m+1}(\Odt{n}),~n=1,\dots, N$, the following holds:
\begin{align}
\label{eq:shnw}
			s_h^n(w,w) &\lesssim h^{2m} \Vert w \Vert_{H^{m+1}(\Odt{n})}^2,\\
\label{eq:shnIw}
			s_h^n(w - \mathcal{I} w, w - \mathcal{I} w) &\lesssim h^{2m} \Vert w \Vert_{H^{m+1}(\Odt{n})}^2.
\end{align}

Consider domains $\cS^n\subset\cO^n\subset \Odt{n}\subset\R^d$ defined as 
\[
\cO^n=\bigcup_{t\in [t_n,t_{n+1}]} \Omega(t)\quad\text{and}\quad \cS^n=\cO^n\setminus\Omega(t_n).
\]
We need the set of all simplices that are intersected by $\cS^n$,
\begin{equation*}
	\T^{n}_+ := \{ T \in \T_h\,: \text{meas}(T,\cS^n)>0\}.
\end{equation*}
We also introduce the set of simplices which are strictly internal for $\Omega^n$, 
\begin{equation*}
\T^{n}_{\rm int} := \{ T \in \T_h\,: T\subset \Omega^n\}.
\end{equation*}

\begin{assumption} \label{ass:overlap} 
	To every element in $\T^n_+\setminus\T^n_{\rm int}$ we require an element in $\T^n_{\rm int}$ that can be reached by repeatedly passing through facets in $\Td{n}$. We assume that there is mapping that maps every element $T \in \T^n_+\setminus\T^n_{\rm int}$ to such a path with the following properties. The number of facets passed through during this path is bounded by $K \lesssim (1 + \frac{\Delta t}{h})$. Furthermore, every uncut element $T \in \T^n_{\rm int}$ is the final element of such a path in at most $M$ of these paths where $M$ is a number that is bounded independently of $n$, $h$ and $\Delta t$.
\end{assumption}

Such a path condition is typical for ghost penalty stabilization methods. In our situation here, for the case $\Delta t \sim h$, it essentially means the following. Every element in the larger set $\T^{n}_+$ can be connected via elements in the active mesh $\Td{n}$ to an element in the smaller set $\T^n_{\rm int}$, and the number of elements needed for this connection (``length of the path'') is uniformly bounded. This kind of condition is not essentially influenced by topological changes of the domain. 

With this assumption one can show that for finite element functions the $L^2$ norm an an extended domain (``the larger set'') can be controlled by the sum of the $L^2$ norm on the original domain (``the smaller set'') and a suitably scaled ghost penalty stabilization  term. The following result is proved in  \cite{lehrenfeld2019eulerian}.
\begin{lemma} \label{lem:gp}
	Under Assumption \ref{ass:overlap}, there holds for $u \in V_h^n$:
	\begin{subequations}
		\begin{align}
			   \norm{u}_{\cO^n}^2 &\lesssim \norm{u}_{\Om{n}{}}^2 + K\, h^2 s_{h}^n(u,u), \label{e:gpglobal}\\
		   \norm{\nabla u}_{\cO^n}^2 &\lesssim \norm{\nabla u}_{\Om{n}{}}^2 + K \,  s_{h}^n(u,u). \label{e:gpglobalgrad}
		\end{align}
	\end{subequations}
\end{lemma}

In the remainder of the error analysis we assume that the Assumptions~\ref{Ass0}-\ref{ass:overlap} are satisfied.  Using \eqref{e:gpglobalgrad} it easily follows that the discretization \eqref{e:unfFEM1} has a unique solution.
\subsection{Stability analysis} \label{s_stab}
A stability estimate is derived in the following theorem. To obtain this result we use the estimate of Lemma~\ref{L:nb}.
\begin{theorem}\label{Th1} Take 
\begin{equation} \label{choicegamma} \gamma_s=c_\gamma \big(1+ \frac{\Delta t}{h}\big)
\end{equation}
 with a constant $c_\gamma$ that is sufficiently large.
The solution of \eqref{e:unfFEM1} satisfies the following estimate:
	\begin{equation}\label{FE_stab1}
			\|u_h^k\|^2_{\Om{k}{}} + {\Delta t} \sum_{n=1}^{k}\left(\norm{\nabla u_h^n}_{\Om{n}{}}^2 + \gamma_s s_h^n(u_h^n,u_h^n) \right)
			\lesssim \norm{u^0}_{\Om{0}{}}^2 +  \Delta t \gamma_s  s_h^0(u_h^0,u_h^0) + \Delta t  \norm{\nabla u_h^0}_{\Om{0}{}}^2, 
	\end{equation}
	for $0 \leq k \leq N$.
\end{theorem}
\begin{proof}
	We test \eqref{e:unfFEM1} with $u_h^n$ and multiply by $2 \Delta t$ which yields:
	\begin{equation*}
		\norm{u_h^n}_{\Om{n}{}}^2 + \norm{u_h^n - u_h^{n-1}}_{\Om{n}{}}^2 + 2\Delta t \, a^n(u_h^n,u_h^n) + 2\Delta t \, \gamma_s s_h^n(u_h^n,u_h^n) = \norm{u_h^{n-1}}_{\Om{n}{}}^2
	\end{equation*}
or simplifying, 
	\[
	\norm{u_h^n}_{\Om{n}{}}^2 +  2\Delta t \left(\norm{\nabla u_h^n}_{\Om{n}{}}^2 + \gamma_s s_h^n(u_h^n,u_h^n) \right) \le  \norm{u_h^{n-1}}_{\Om{n}{}}^2. 
\]
Applying estimate \eqref{est:nb} to the term $\norm{u_h^{n-1}}_{\Om{n}{}}^2$ we get
\[
	\begin{split}
	& \norm{u_h^n}_{\Om{n}{}}^2 +  2\Delta t \left(\norm{\nabla u_h^n}_{\Om{n}{}}^2 + \gamma_s s_h^n(u_h^n,u_h^n) \right)  \\ 
	&\le
\norm{u_h^{n-1}}_{\Om{n-1}{}}^2 + C\,\Delta t \left( (1+\eps^{-1})\norm{u_h^{n-1}}_{\cO^{n-1}}^2 +\eps\norm{\nabla u_h^{n-1}}_{\cO^{n-1}}^2\right).
	\end{split}
\]
Using further \eqref{e:gpglobalgrad}, we obtain
\[
\begin{split} & \norm{u_h^n}_{\Om{n}{}}^2 +  2\Delta t \big(\norm{\nabla u_h^n}_{\Om{n}{}}^2 + \gamma_s s_h^n(u_h^n,u_h^n) \big)  \le
\norm{u_h^{n-1}}_{\Om{n-1}{}}^2   \\ & +  C\,\Delta t \left( (1+\eps^{-1})\norm{u_h^{n-1}}_{\Om{n-1}{}}^2 +\eps\norm{\nabla u_h^{n-1}}_{\Om{n-1}{}}^2+ K(\eps+(1+\eps^{-1})h^2) \rev{s_h^{n-1}}(u_h^{n-1},u_h^{n-1})\big)\right).
\end{split}
\]
We take $\eps= C^{-1} $. For a suitable constant $\hat c$ we have, cf. Assumption~\ref{ass:overlap}, $C K(\eps+(1+\eps^{-1})h^2) = K( 1 +C(1+C)h^2)  \leq \hat c (1+\frac{\Delta t}{h})$. We take $c_\gamma \geq \hat c$ and thus get $C K(\eps+(1+\eps^{-1})h^2) \leq c_\gamma(1+\frac{\Delta t}{h})=\gamma_s$. Using this we obtain with $\tilde C=C(1+C)$:
\begin{equation}\label{aux379}
	\begin{split}
		\norm{u_h^n}_{\Om{n}{}}^2 &+  2\Delta t \big(\norm{\nabla u_h^n}_{\Om{n}{}}^2 + \gamma_s s_h^n(u_h^n,u_h^n) \big)  \\ 
		&\le
		\norm{u_h^{n-1}}_{\Om{n-1}{}}^2 + \tilde C\Delta t \norm{u_h^{n-1}}_{\Om{n-1}{}}^2 +\Delta t \norm{\nabla u_h^{n-1}}_{\Om{n-1}{}}^2+  \Delta t \, \gamma_s s_h^{n-1}(u_h^{n-1},u_h^{n-1}).
	\end{split}
\end{equation}
	Summing up  over $n=1,\dots,k,~k\leq N$, yields
	\[ \begin{split}
		& \norm{u_h^k}_{\Om{k}{h}}^2 +\Delta t  \sum_{n=1}^k  \norm{\nabla u_h^n}_{\Om{n}{}}^2 + \Delta t \gamma_s  \sum_{n=1}^k s_h^n(u_h^n,u_h^n)
		\\ &  \leq \norm{u^0_h}_{\Om{0}{}}^2 + \tilde C\Delta t\sum_{n=0}^{k-1} \| u^n_h \|_{\Om{n}{}}^2 + \Delta t \gamma_s   s_h^0(u_h^0,u_h^0) + \Delta t  \norm{\nabla u_h^0}_{\Om{0}{}}^2.
	\end{split}
	\]
	We apply Gronwall's Lemma to obtain 
	\[ 
	\norm{u_h^k}_{\Om{k}{h}}^2 +\Delta t  \sum_{n=1}^k  \norm{\nabla u_h^n}_{\Om{n}{}}^2 + \Delta t \gamma_s  \sum_{n=1}^k s_h^n(u_h^n,u_h^n)
	\leq e^{k \tilde C\Delta t}\left(\norm{u^0}_{\Om{0}{}}^2 +  \Delta t\gamma_s  s_h^0(u_h^0,u_h^0) + \Delta t  \norm{\nabla u_h^0}_{\Om{0}{}}^2\right),
\]
which completes the proof. 
\end{proof}

\subsection{Consistency estimates} \label{sec:consist}
In this section we derive bounds for the consistency error. 
In the analysis we need a smooth  extension of the exact solution $u$ to a space--time neighborhood
\[
\cO(\cQ) = \bigcup\limits_{t \in (0,T)} \cO(\Omega(t)) \times \{t\},\quad  \cQ\subset\cO(\cQ)\subset \Bbb{R}^{d+1},
\]
where $\cO(\Omega(t))$ are sufficiently large spatial neighborhoods of $\Omega(t)$ such that these contain the computational domains $\Omega_{\delta,h}^n$ used in the discretization method \eqref{defV}--\eqref{e:unfFEM1}.
We use that the space--time domain $\cQ$ is Lipschitz (Assumption~\ref{Ass4}), so the solution $u$ can be smoothly extended from $\cQ$ to  $\cO(\cQ)$. 
More precisely, there is a linear extension operator $\E{}:L^2(\cQ)\to L^2(\cO(\cQ))$ such that ~\cite[Section~VI.3.1]{stein2016singular}:
\begin{equation}\label{ExtBound0}
	\Vert {\E{}} u \Vert_{W^{k,p}(\mathcal{O}(\cQ))} \leq  C \Vert u \Vert_{W^{k,p}(\cQ)},\quad \text{for}~u \in W^{k,p}(\cQ),~~k=0,\dots,m+2,~~1\le p\le\infty.
\end{equation}
 We will identify $u$ with its extension.

Testing \eqref{transport} at $t=t_n$ with $v_h \in V_h^n$ we find that any smooth solution to \eqref{transport} satisfies
\begin{equation} \label{e:exid}
	\int_{\Om{n}{}} \partial_t u (t_n) v_h\, dx + a^n(u(t_n),v_h) = 0 \quad \text{ for all } ~v_h \in V_h^n.
\end{equation}
Given the extension, the error $\err^n = u^n - u_h^n$ is well defined in $\Odt{n}$. Subtracting \eqref{e:unfFEM1} from \eqref{e:exid} we obtain the error equation
\begin{equation} \label{e:erreq}
	\int_{\Om{n}{}} \frac{\err^n - \err^{n-1}}{\Delta t} v_h \, dx +  a_h^n(\err^n,v_h) + \gamma_s s_h^n(\err^n,v_h) = \consist(v_h),
\end{equation}
with\vspace*{-2ex}
\begin{align*}
	\consist(v_h) :=
	& \hphantom{+} \overbrace{
		\int_{\Om{n}{}} u_t(t_n) v_h dx - \int_{\Om{n}{}} \frac{u^n-u^{n-1}}{\Delta t} v_h dx
	}^{I_1}
	+ \overbrace{ \vphantom{\int_{\Om{n}{}}} \gamma_s s_h^n(u^n,v_h)
	}^{I_2}.
\end{align*}

\begin{lemma}\label{l_consist} Assume $u\in W^{2,\infty}(\cQ)\cap H^{m+2}(\cQ)$,
	then the consistency error has the bound
	\begin{equation}\label{est:consist}
		|\consist(v_h)|\lesssim (\Delta t+\gamma_s^\frac12 h^m  ) \, (\norm{u}_{W^{2,\infty}(\cQ)}+\norm{u}_{H^{m+2}(\cQ)}) \,
		(\|v_h\|_{\Om{n}{}}+ \gamma_s^\frac12 s_h^n(v_h,v_h)^\frac12).
	\end{equation}
\end{lemma}
\begin{proof} We treat $\consist(v_h)$ term by term, starting with $I_1$:
	\[
	I_1=  \int_{\Om{n}{}}\int_{t_{n-1}}^{t_n}\frac{t-t_{n-1}}{\Delta t} u_{tt} \,dt\,v_h\, dx.
	\]
	We have 
	\begin{equation} \label{h17}
			\left| \int_{\Om{n}{}}\int_{t_{n-1}}^{t_n}u_{tt} \frac{t-t_{n-1}}{\Delta t}\,dt\,v_h\, ds\right| \le \tfrac12 \Delta t\|u_{tt}\|_{L^\infty(\cO(\cQ))}\|v_h\|_{L^1(\Om{n}{})} \lesssim \Delta t\|u\|_{W^{2,\infty}(\cQ)}\|v_h\|_{\Om{n}{}}.
	\end{equation}
	For the $s_h^n(u^n,v_h)$ part in the  term $I_2$ we  use the Cauchy--Schwarz inequality and  the result in \eqref{eq:shnw},
	\begin{equation*}
		s_h^n(u^n,v_h)\le s_h^n(u^n,u^n)^\frac12 s_h^n(v_h,v_h)^\frac12 \lesssim h^{m}\|u\|_{H^{m+1}(\Odt{n})}s_h^n(v_h,v_h)^\frac12
		\lesssim h^{m}\|u\|_{H^{m+2}(\cQ)}s_h^n(v_h,v_h)^\frac12.
	\end{equation*}
	In the last estimate we used the trace inequality $\|u\|_{H^{m+1}(\Odt{n})}\lesssim \|u\|_{H^{m+2}(\mathcal{O}(\cQ))}$ and \eqref{ExtBound0}. Combining this bound with the one in \eqref{h17}  yields the bound  \eqref{est:consist}.
\end{proof}

\begin{remark} \label{remStabincons} \rm
The factor $\gamma_s^\frac12 h^m$ in the  consistency error bound results from the fact that the ghost penalty stabilization $s_h(\cdot,\cdot)$ that we use is \emph{not} consistent, $s_h(u,v_h) \neq 0$. The stabilization parameter $\gamma_s=c_\gamma(1 +\frac{\Delta t}{h})$ is not uniformly bounded if $\Delta t \sim h^\alpha$ with $\alpha < 1$. At the same time, the latter scaling is not very realistic. This consistency error due to the stabilization can be avoided if one uses a consistent variant of ghost penalty stabilization, i.e., one for which $s_h(u,v_h) = 0$ for all $v_h$ holds,  e.g., the derivative jump formulation \cite{Badia2022}. \rev{In the latter case, the method does not impose any restrictions on $\Delta t$ in terms of $h$. }
\end{remark}

\subsection{Error estimate in the energy norm} \label{sec:aprioriest}
We let $u^n_I = \mathcal{I} u^n \in V_h^n$ be the Lagrange interpolant  for $u^n$ in $\Odt{n}$; we assume $u^n$  sufficiently smooth so that the interpolation is well-defined. Following standard lines of argument, we split $\err^n$ into finite element and approximation parts,
\[
\err^n=\underset{\mbox{$e^n$}}{\underbrace{(u^n-u^n_I)}}+\underset{\mbox{$e^n_h \in V_h^n$}}{\underbrace{(u^n_I-u^n_h)}}.
\]
Equation \eqref{e:erreq} yields
\begin{equation}\label{e:err1}
	\int_{\Om{n}{}}\left(\frac{e^n_h-e^{n-1}_h}{\Delta t} \right) v_h\,ds+ a^n(e_h^n,v_h)+\gamma_s s_h^n(e_h^n,v_h) = \interpol(v_h)+\consist(v_h),\quad \forall~v_h\in V^n_h,
\end{equation}
with the interpolation term
\[
\interpol(v_h)=-\int_{\Om{n}{}}\left(\frac{e^n-e^{n-1}}{\Delta t} \right) v_h\,ds_h-  a^n(e^n,v_h)-\gamma_s s_h^n(e^n,v_h).
\]
We give the estimate for the interpolation terms in the following lemma.

\begin{lemma}\label{l_interp} Assume $u\in H^{m+2}(\cQ)$. 
	The following  holds:
	\begin{equation}\label{est_inter}
		|\interpol(v_h)|\lesssim \gamma_s^\frac12 h^m \,\|u\|_{H^{m+2}(\cQ)} \, 	(\|v_h\|_{H^1(\Om{n}{})}+ \gamma_s^\frac12 s_h^n(v_h,v_h)^\frac12).
	\end{equation}
\end{lemma}
The proof of Lemma~\ref{l_interp} repeats the arguments from \cite[Lemma~5.7]{lehrenfeld2019eulerian}.

Combining the results we obtain an error bound using standard arguments, which we include for completeness.
\begin{theorem}\label{Th2}
	Assume that the  solution $u$ to \eqref{transport}  has smoothness $u\in W^{2,\infty}(\cQ)\cap H^{m+2}(\cQ)$. 
	For  $u_h^n$, $n=1,\dots,N$, the finite element solution of \eqref{e:unfFEM1},
	define $\err^n = u_h^n - u^n$. For    $\Delta t$ sufficiently small the following error estimate holds:
	\begin{equation}\label{FE_est1}
		\|\err^n\|^2_{\Om{n}{}}+\Delta t \sum_{k=1}^{n}\! \left( \|\nabla\err^k \|^2_{\Om{k}{}}
		+\gamma_s s_h^n(\err^n,\err^n)\right)
		\lesssim  \exp(c_0 t_n) R(u) (\Delta t^2+ \gamma_s h^{2m} ),
	\end{equation}
	with $R(u) := \|u\|_{H^{m+2}(\cQ)}^2+\|u\|_{W^{2,\infty}(\cQ)}^2$ and $c_0$ independent of $h$, $\Delta t$, $n$ and of the positions of $\Omega^n$ in the background mesh.
\end{theorem}
\begin{proof}
	We set $v_h=2\Delta t e^n_h$ in \eqref{e:err1}. This gives
	\begin{equation*}
		\|e_h^n\|^2_{\Om{n}{}}- \|e_h^{n-1}\|^2_{\Om{n}{}} +\|e_h^n-e_h^{n-1}\|^2_{\Om{n}{}}+{2\Delta t}  a^n(e_h^n,e_h^n)+{2\Delta t}\gamma_s s^n_h(e_h^n,e_h^n)=2\Delta t\big(\interpol(e_h^n)+\consist(e_h^n)\big).
	\end{equation*}
	We repeat the arguments as in the proof of Theorem~\ref{Th1} for estimating the term $\|e_h^{n-1}\|^2_{\Om{n}{}}$ by $\|e_h^{n-1}\|^2_{\Om{n-1}{}}$.  This yields
	\begin{align}\label{aux1}
		&\norm{e_h^k}_{\Om{k}{}}^2 + \Delta t  \sum_{n=1}^k  \norm{\nabla e_h^n}_{\Om{n}{}}^2 + \Delta t \gamma_s \sum_{n=1}^k s_h^n(e_h^n,e_h^n) \\
		& \leq \norm{e^0_h}_{\Om{0}{}}^2 + \tilde C \Delta t \sum_{n=0}^{k-1} \| e^n_h \|_{\Om{n}{}}^2 + \Delta t \gamma_s  s_h^0(e_h^0,e_h^0) + \Delta t  \norm{\nabla e_h^0}_{\Om{0}{}}^2
		+ 2\Delta t  \sum_{n=1}^k(\interpol(e_h^n)+\consist(e_h^n)). \nonumber
	\end{align}
	To estimate the interpolation and consistency terms, we apply Young's inequality to the right-hand sides of \eqref{est:consist} and \eqref{est_inter} and thus obtain
	\begin{align*}
		2 \Delta t (\consist(e_h^n) + \interpol(e_h^n))&
		\le c\,\Delta t (\Delta t^2+ \gamma_s h^{2m}) R(u)
		+\frac{\Delta t}{2}
		\left(\|e_h^n\|_{\Om{n}{}}^2+\|\nabla e_h^n\|_{\Om{n}{}}^2 +\gamma_s s_h^n(e_h^n,e_h^n) \right),
	\end{align*}
	with a constant $c$ independent of $h$, $\Delta t$, $n$ and of the position of the domain in the background mesh. Substituting this in \eqref{aux1} and  noting $e^0_h=0$ in $\Odt{0}$ we get
	\[
		 \norm{e_h^k}_{\Om{k}{}}^2  + \tfrac12 \left(  \, \Delta t  \sum_{n=1}^k  \norm{\nabla e_h^n}_{\Om{n}{}}^2 + \gamma_s \Delta t \sum_{n=1}^k s_h^n(e_h^n,e_h^n) \right)
		 \leq C \Delta t \sum_{n=0}^{k} \| e^n \|_{\Om{n}{}}^2 +cR(u) (\Delta t^2+ \gamma_s h^{2m}).
	\]
	We apply the discrete Gronwall inequality to get
	\begin{align}
		\|e_h^k & \|^2_{\Om{k}{}}+\tfrac12 \Delta t \sum_{n=1}^{k} \left( \|\nabla e_h^n\|^2_{\Om{n}{}}
		+\gamma_s s_h^n(e_h^n,e_h^n)\right)
		\lesssim \exp(C t_k) R(u)(\Delta t^2+ \gamma_s h^{2m} ).
		\nonumber
	\end{align}
	Now the triangle inequality, interpolation estimates, \rev{and \eqref{eq:shnIw}} give
	\begin{equation*}
		\begin{split}
			\|\err^k\|^2_{\Om{k}{}}&+\tfrac12 \Delta t \sum_{n=1}^{k} \left( \|\nabla\err^n\|^2_{\Om{n}{}}
			+\gamma_s s_h^n(\err^n,\err^n)\right)
			\\ &
			\leq
			\exp(C t_k) R(u)(\Delta t^2+ \gamma_s h^{2m} ) +\|e^k\|^2_{\Om{k}{}}+\tfrac12 \Delta t\sum_{n=1}^{k} \left( \|\nabla e^n\|^2_{\Om{n}{}}
			+\gamma_s s_h^n(e^n,e^n)\right)
			\\ &
			\lesssim \exp(C t_k) R(u)(\Delta t^2+\gamma_s h^{2m} ) + \norm{u}_{H^{m+2}(\cQ)} \gamma_s h^{2m}.
		\end{split}
	\end{equation*}
	This completes the proof.
\end{proof}

\section{Numerical example}\label{s:numer}
We present numerical  results for a simple example of a two-dimensional  level set domain  with a topological 
change induced by a smooth level set function with an isolated critical point with $\frac{\partial \phi}{\partial t}(x_c,t_c) > 0$, cf. Section~\ref{newSectdomains}.  The results are obtained using  NGSolve/netgen, cf. \cite{ngsolve} with the add-on ngsxfem, cf. \cite{ngsxfem}. In the implementation we use a parametric finite element technique, cf. Remark~\ref{remimplementation}.  

The evolving domain $\Omega(t)$ is defined as the subzero level of the smooth level set function
\begin{equation}\label{phi}
\phi(x,t)=\phi(x_1,x_2,t) = (t - 0.25) - \big( (x_1^2-0.3 x_1^4) -x_2^2\big)\quad \text{for}~~t\in[0,\,0.5].
\end{equation}
This subzero level describes the evolution of  a smooth domain  that splits into two   domains at   $(x_c,t_c)=(\mathbf{0},0.25)$, which  is a nondegenerate critical point with  $\frac{\partial \phi}{\partial t}(x_c,t_c) > 0$. 
The given smooth solution to the heat equation in $\Omega(t)$  was set to be 
\[
u(x,t)= \phi(x,t)^2 + (x_1+x_2+1). 
\]
This solution satisfies  \eqref{transport} with  a non-zero right-hand side and non-zero flux on $\Gamma(t)$, which required obvious modifications of the method.

 In the finite element method we use a quasi-uniform triangulation of the bulk domain   $\wOm =(-2,2)\times (-1.5,1.5)$, obtained by making $L_x$ levels of uniform refinements of an initial triangulation with $h_0=0.5$.  
In all experiments the time step is taken as $\Delta t =0.1\,2^{-L_t}$, with a specified  $L_t=1,2,3,\dots$.

We report convergence results for the discretization method \eqref{e:unfFEM1}, i.e., with time discretization based on BDF1,   with linear ($m=1$) or quadratic ($m=2$) finite elements. We set $L_t=2L_x$. Results are presented in  Figure~\ref{BDF1}, where errors are reported in the time-discrete analogues of $L^\infty(L^{2})$ and  $L^2(H^1)$ norms.
\begin{figure}[ht!]\centering
	\includegraphics[width=0.42\textwidth]{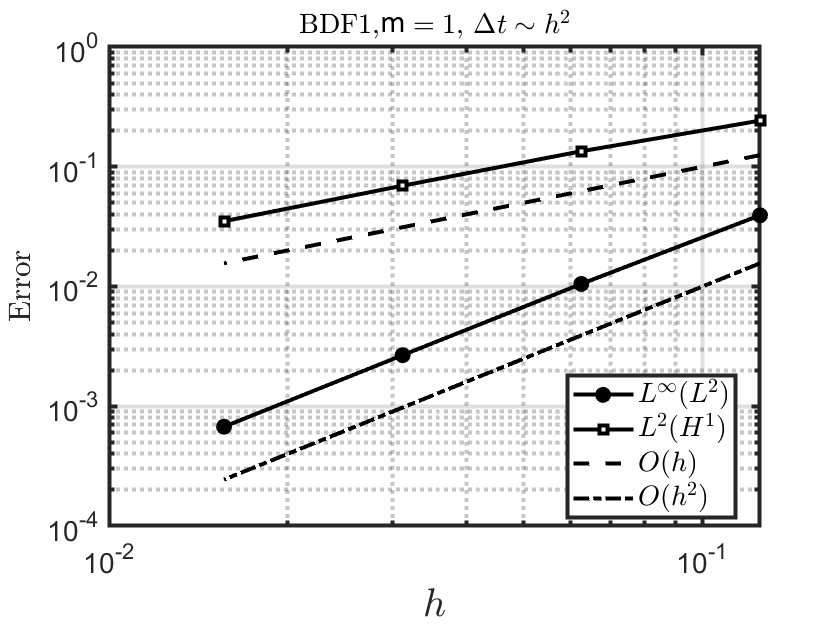}
	\includegraphics[width=0.42\textwidth]{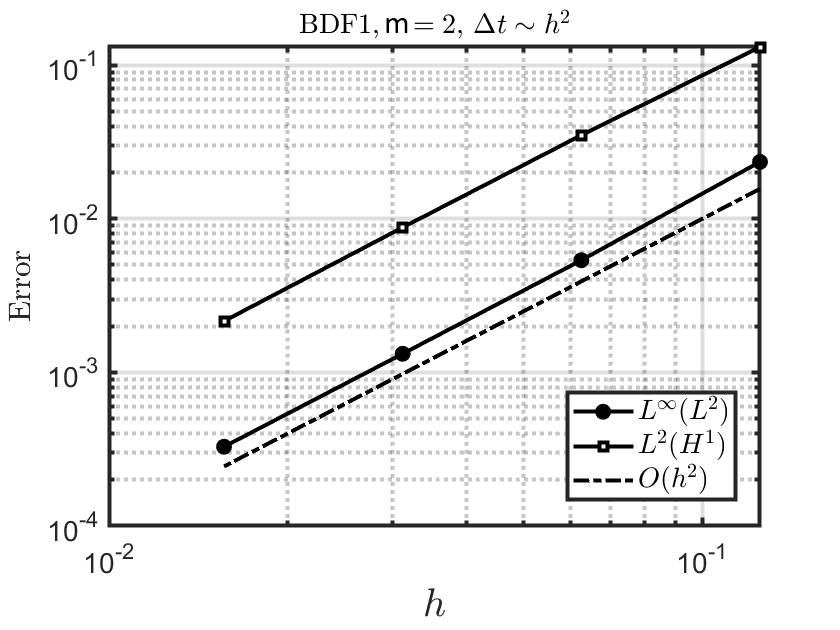}
\caption{\label{BDF1}BDF1 results for  $m=1$ (left panel) and $m=2$ (right panel).}
\end{figure}

The numerical results  clearly confirm the convergence order in the $L^2(H^1)$ norm predicted by Theorem~\ref{Th2}. Moreover, we see that we are regaining an extra convergence order in the $L^\infty(L^{2})$ norm.

With obvious modifications (cf. also \cite{lehrenfeld2019eulerian})  one obtains a BDF2 time stepping variant of the method \eqref{e:unfFEM1}. In this variant we also use  finite element spaces with $m=1$ and $m=2$, but now  set  $L_t=L_x$. Results are presented in Figure~\ref{BDF2}. 
\begin{figure}[ht!]
\centering
	\includegraphics[width=0.42\textwidth]{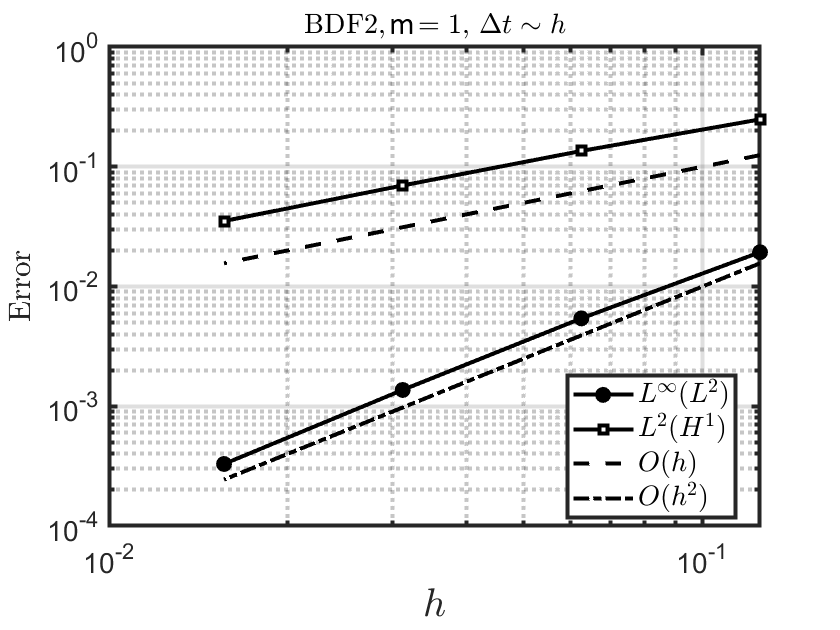}
	\includegraphics[width=0.42\textwidth]{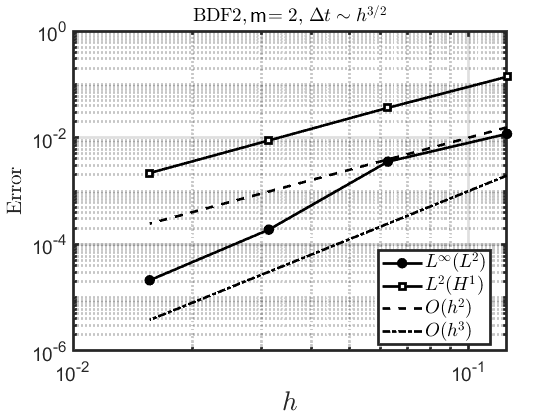}
	\caption{\label{BDF2} BDF2 results for FE orders $m=1$ (left panel) and $m=2$ (right panel).}
\end{figure}

For this BDF2 variant we also observe optimal order of convergence.

Figure~\ref{solution} visualizes $\Omega(t)$ for several time instances. The figure also shows the active numerical domains and the numerical solution extended to the numerical domain.

\begin{figure}[ht!]\label{solution}\centering
	\begin{minipage}{0.9\textwidth}\centering
	\includegraphics[width=0.29\textwidth]{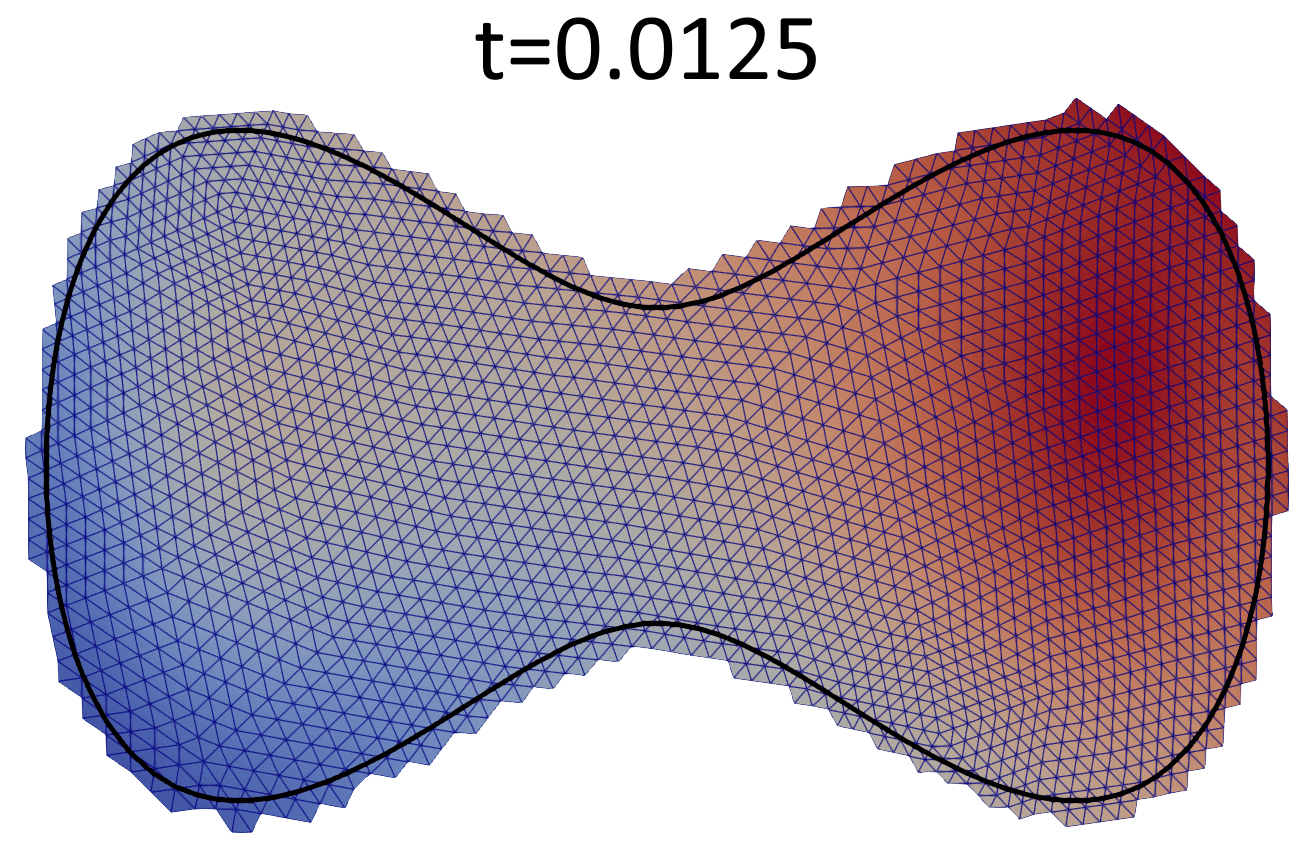}\quad	
	\includegraphics[width=0.31\textwidth]{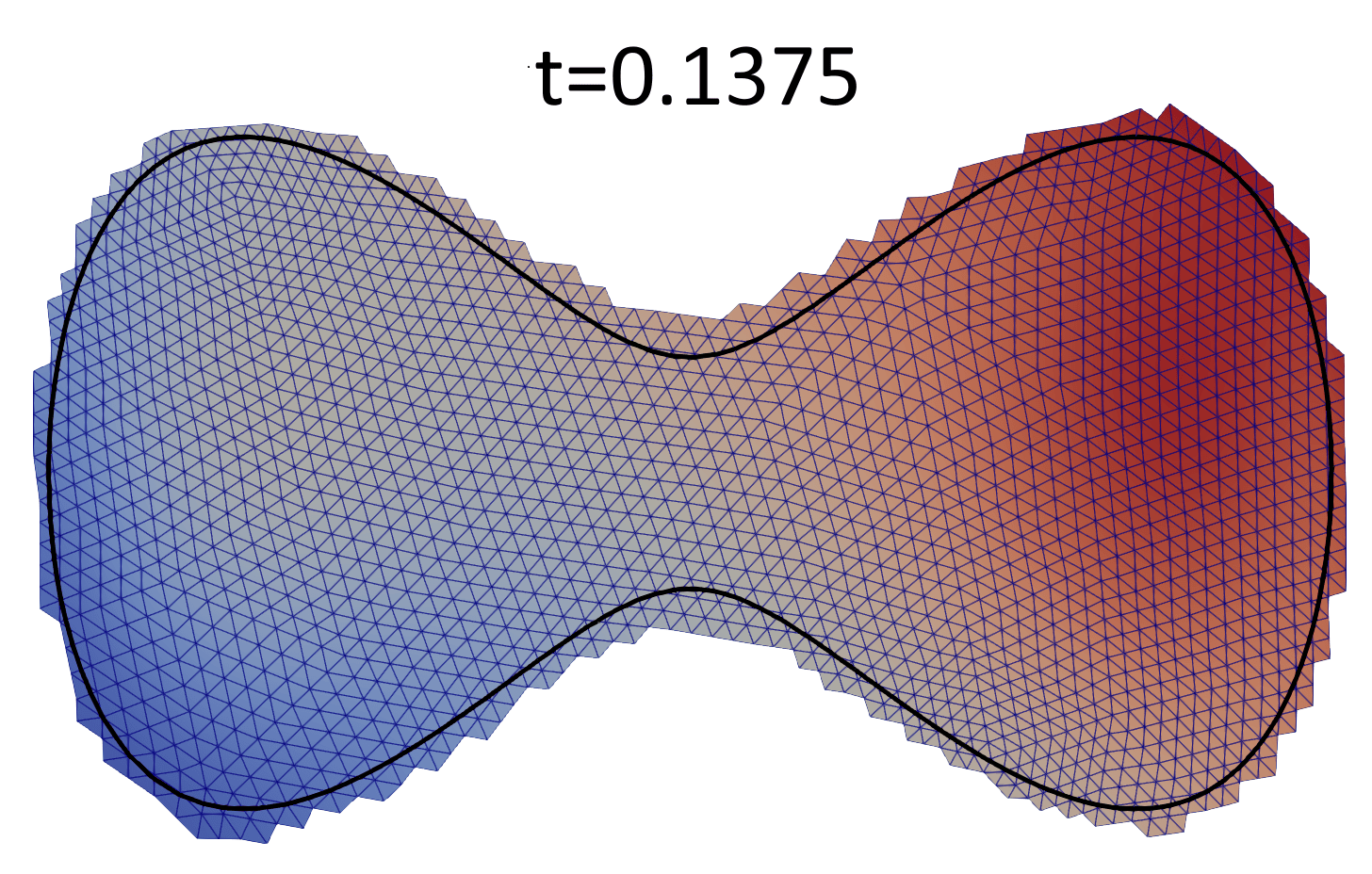}\quad
	\includegraphics[width=0.31\textwidth]{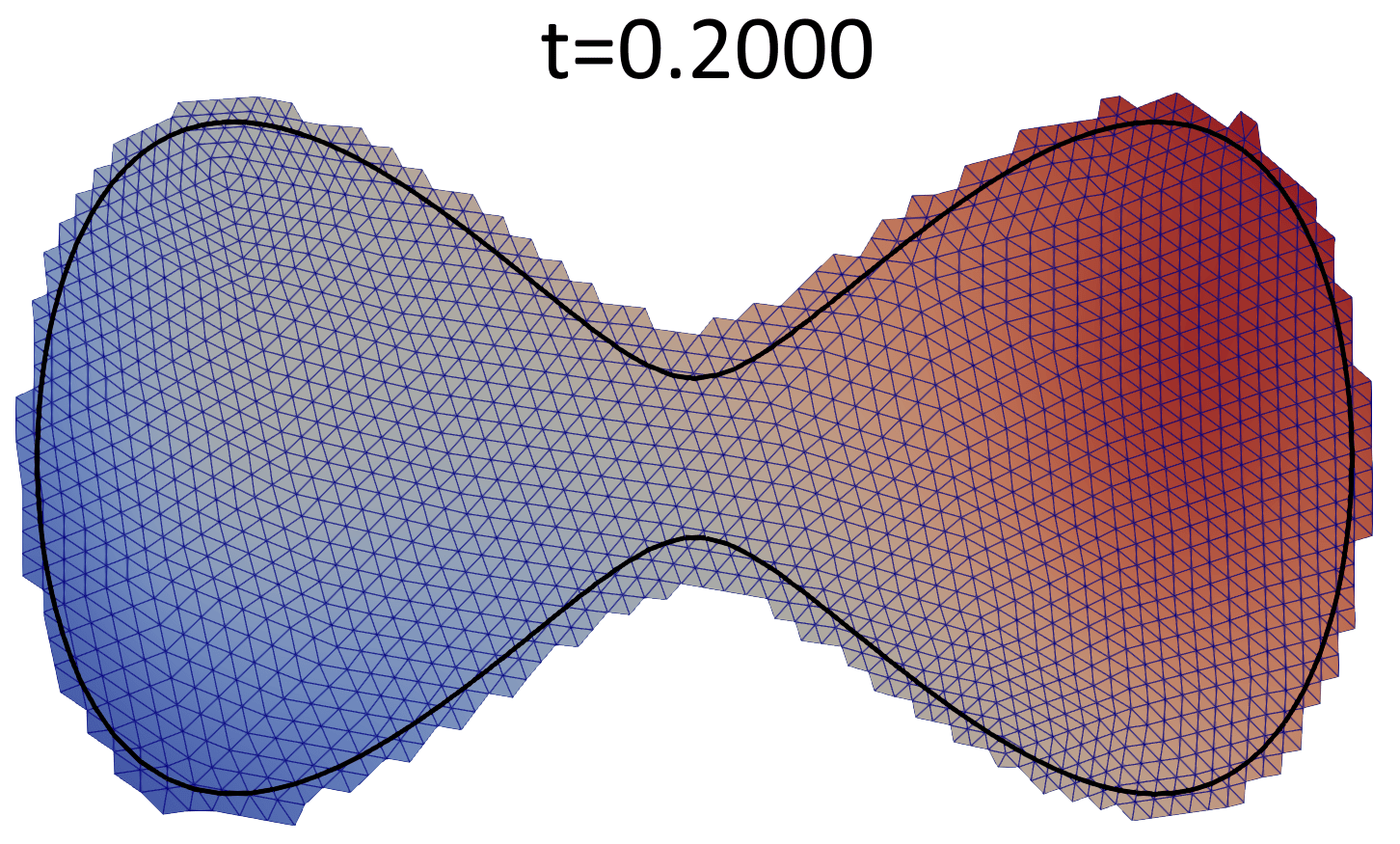}\\[2ex]
	\includegraphics[width=0.30\textwidth]{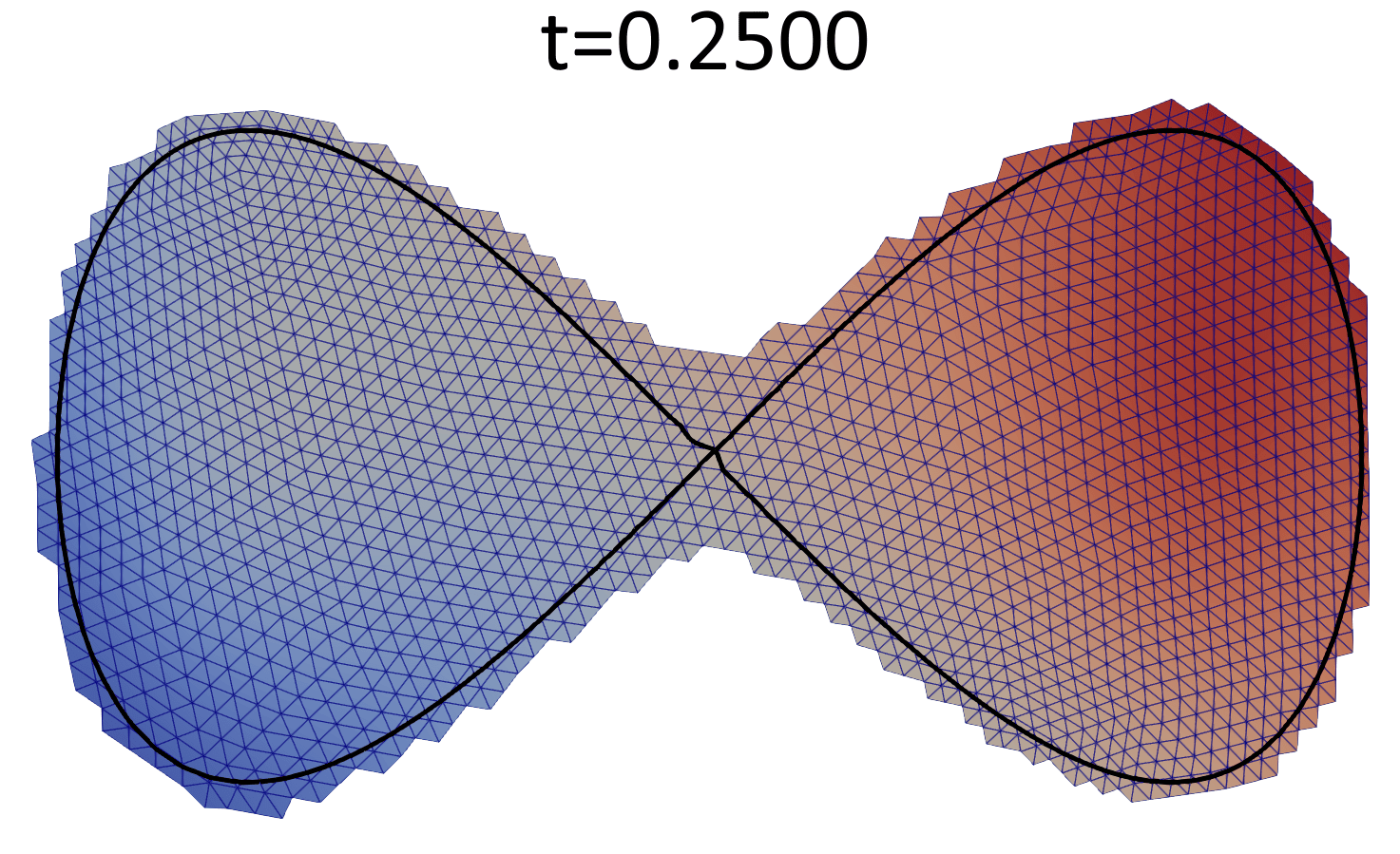}\quad
	\includegraphics[width=0.31\textwidth]{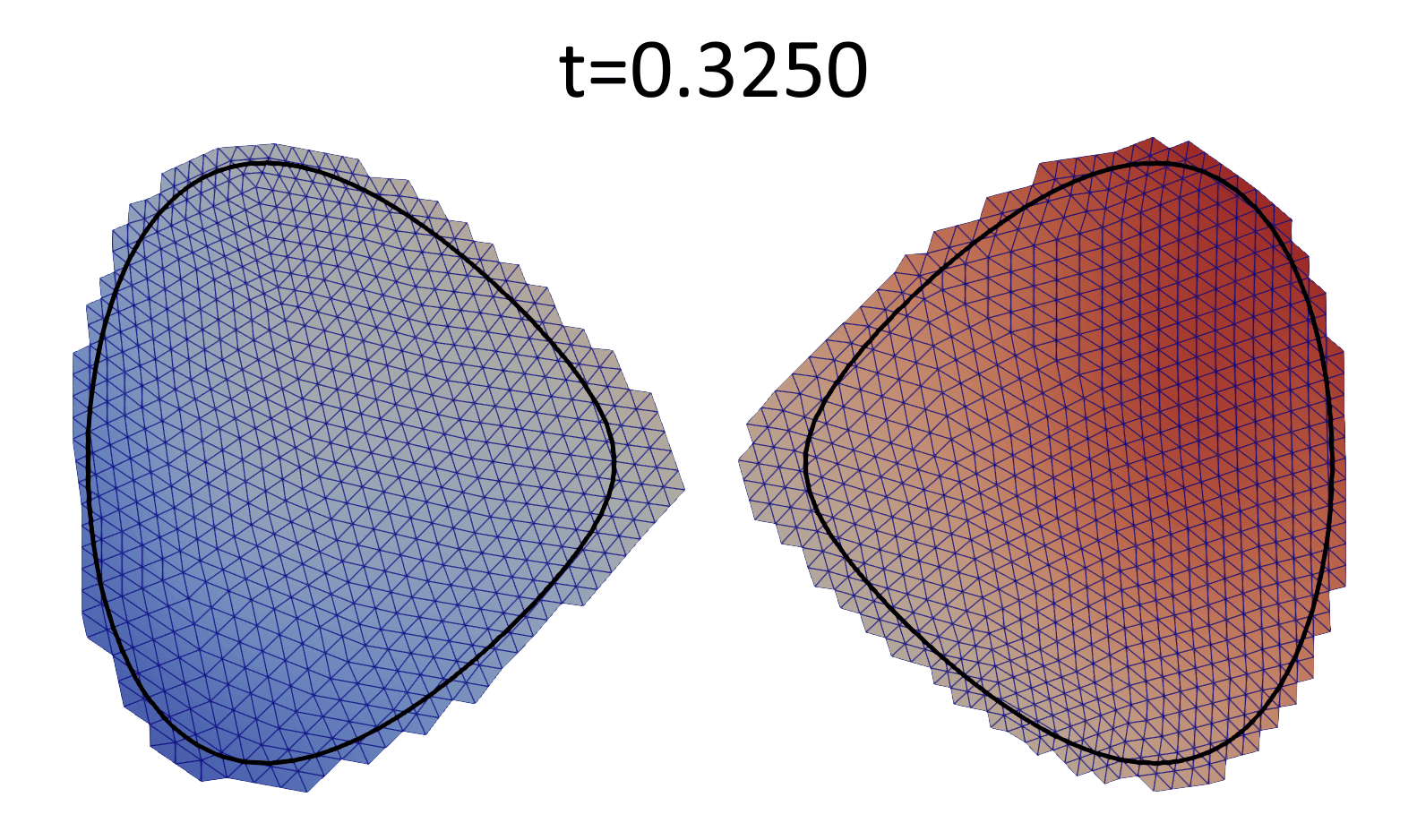}\quad
	\includegraphics[width=0.31\textwidth]{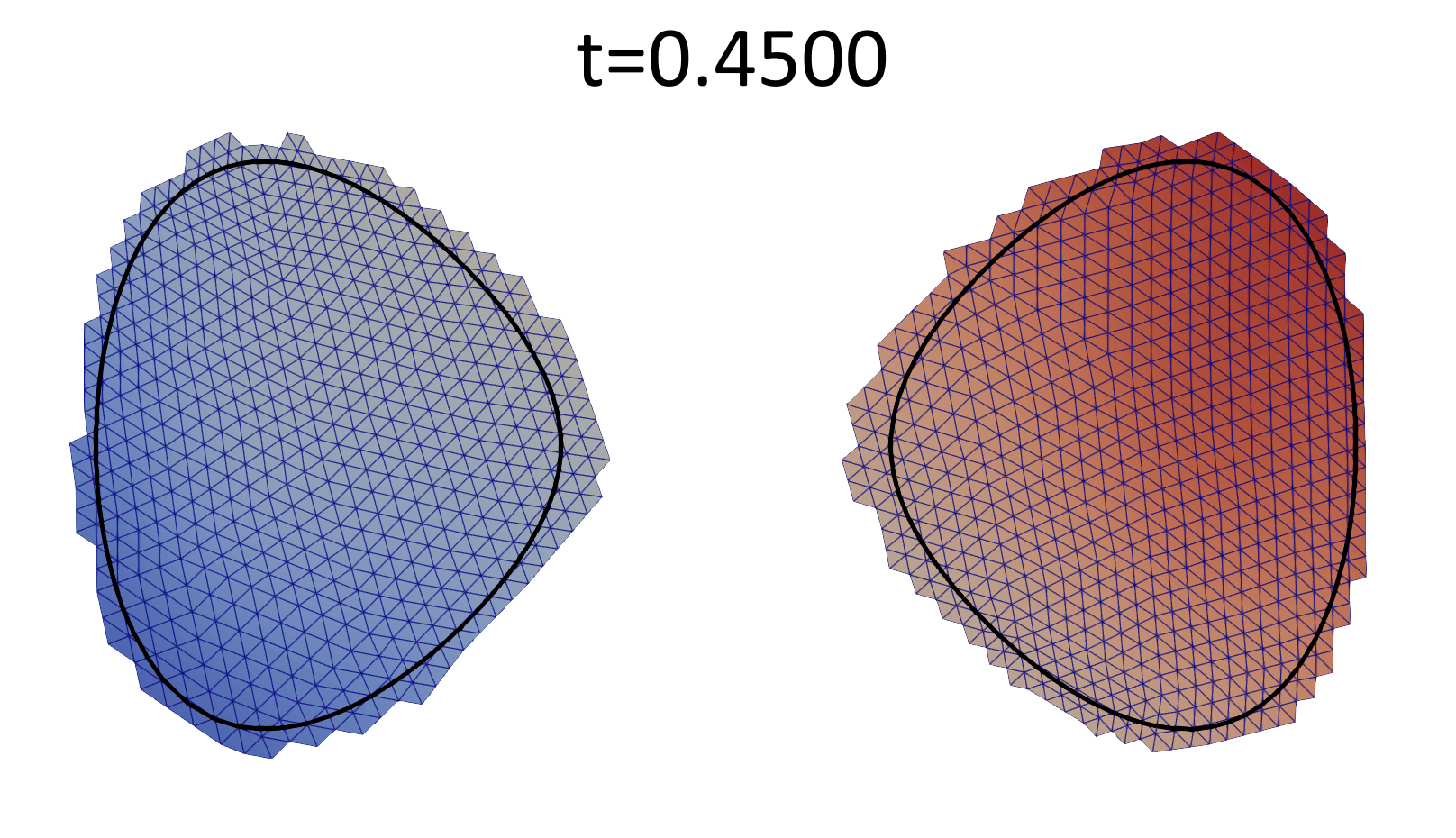}
	\end{minipage}
	\begin{minipage}{0.09\textwidth}
	\includegraphics[width=\textwidth]{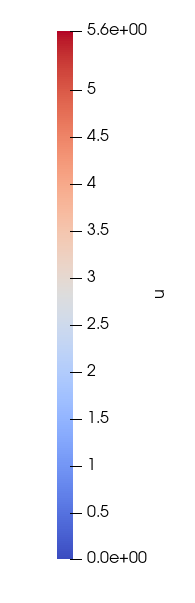}
	\end{minipage}
	\caption{Visualization of the splitting domain and the numerical solution with $L_x=3$. The boundary $\G(t)$ is shown by the black contour line.}
\end{figure}

We repeated the experiment with reversed time in the definition of $\phi$, i.e. $t-0.25$ was replaced by $0.25-t$ in \eqref{phi}. This  corresponds to the case of merging domains. 
Although this case is not covered by the available analysis, we observed convergence results very similar to those presented in Figures~\ref{BDF1}--\ref{BDF2}. An error analysis for such a case with merging level set domains is left for future research.

\section{Conclusions and outlook}
Topological transitions in 2D and 3D time-dependent domains can occur through various scenarios. In this paper, we identified specific scenarios that are amenable to rigorous stability and error analysis of an Eulerian unfitted finite element method for solving linear parabolic problems posed on such domains. To this end, we formulated several structural assumptions regarding topology changes and investigated when these assumptions hold or fail for a class of level set domains. A connection to Morse theory enabled us to identify certain admissible topological changes, such as domain splitting, island vanishing, and hole creation, as special cases.

The structural assumptions were primarily driven by  analysis requirements of the specific Eulerian unfitted finite element method considered in this paper. For this  method we derived an optimal order error bound. We  expect that the analysis of alternative computational methods, for example  those based on space–time variational formulations, may accommodate a different or broader class of topological changes. Nevertheless, the results presented here are the first of their kind and show that the rigorous numerical analysis of PDEs on domains undergoing topological transitions is an achievable goal.

Much work remains to be done concerning rigorous analysis of  well-posedness and regularity of solutions of PDEs posed on domains undergoing topology changes, particularly in relation to how the transition occurs. In this paper, certain regularity properties of the solution were assumed without proof.
Finally, a deeper understanding of how the transition scenarios considered in this work and other ones  correspond to real-world phenomena (e.g., cell fusion or the breakup of viscous fluid droplets) would be highly valuable.

\section*{Acknowledgment} The author Maxim Olshanskii was partially supported by National Science Foundation under the grant DMS-2408978. Arnold Reusken thanks the German Research Foundation (DFG) for financial support within the Research Unit ``Vector- and tensor valued surface PDEs'' (FOR 3013) with project no. RE 1461/11-2.
\bibliographystyle{siam}
\bibliography{literatur}{}

\appendix

\section{Stabilization form}	\label{A1}
\ \\ 
The computational domain (in time step $n$) is formed by the triangulation $\Td{n}$, cf.~\eqref{def5}. The stabilization is defined on a corresponding  boundary strip, defined as:
\begin{equation}
	\TS{n} :=\{ T \in \Td{n}\, : \, \dist(x,\Gamma^n) \leq \delta \text{ for some } x \in T \} .
\end{equation}
The boundary strip includes  elements intersected by the boundary of $\Omega^n$, but possibly also some elements that are completely inside or outside of $\Om{n}{}$. We define the set of facets  in $\Td{n}$ and $\TS{n}$:
\begin{equation} \label{e:deffhk}
	\Fh^n := \{ \overline{T_1} \cap \overline{T_2}\, :\, T_1 \in \Td{n},~ T_2 \in \TS{n}, T_1\neq T_2,~ \text{meas}_{d-1} (\overline{T_1} \cap \overline{T_2}) > 0 \}.
\end{equation}

For $F \in \Fh^n$ let $\omega_F$ be the facet patch, i.e. $\omega_F = T_1 \cup T_2$ for $T_1$ and $T_2$ as in the definition \eqref{e:deffhk}. We define for $u,v \in V_h^n$
\begin{equation}
	s_h^{n}(u,v) := \gamma_s\sum_{F \in \Fh^n} s_{h,F}^{n}(u,v) \quad \text{with} \quad s_{h,F}^{n}(u,v):= \frac{1}{h^2} \int_{\omega_F} (u_1 - u_2) (v_1 - v_2) dx,
\end{equation}
where $u_1 = \mathcal{E}^P  u|_{T_1}$, $u_2 = \mathcal{E}^P  u|_{T_2}$ (and similarly for $v_1$, $v_2$) where $\mathcal{E}^P: P_m(T) \rightarrow P_m(\R^d)$ is the canonical extension of a polynomial to $\R^d$. 
For the analysis, we also define $s_h^{n}(u,v)$ for arbitrary functions $u,v  \in L^2(\Odt{n})$. In this case, we set $u_1 = \mathcal{E}^P \Pi_{T_1} u|_{T_1}$, $u_2 = \mathcal{E}^P \Pi_{T_2} u|_{T_2}$
where $\Pi_{T_i}$ is the $L^2(T_i)$-orthogonal projection into $P_m(T_i),~i=1,2$. We notice that for $v \in V_h^n$, $\Pi_{T_i} v|_{T_i} = v|_{T_i}$.
The stabilization parameter is taken as
\begin{equation} \label{e:cgamma}
	\gamma_s = \gamma_s(h,\delta_h) = c_{\gamma} \, (1+\frac{\Delta t}{ h}) \ \text{ with } \ c_\gamma>0 \text{ independent of } \Delta t \text{ and } h.
\end{equation}
The choice of this scaling results from Theorem~\ref{Th1}.
\end{document}